\documentclass{amsart}
\usepackage{latexsym,amsxtra,amscd,ifthen}
\usepackage{amsfonts}
\usepackage{color,graphicx}
\usepackage{verbatim}
\usepackage{amsmath}
\usepackage{amsthm}
\usepackage{amssymb}
\usepackage{url}

\usepackage{tikz}
\usetikzlibrary{arrows,shapes,chains}

\theoremstyle{plain}

\newtheorem{theorem}{Theorem}
\newtheorem{lemma}[theorem]{Lemma}
\newtheorem{proposition}[theorem]{Proposition}
\newtheorem{corollary}[theorem]{Corollary}

\numberwithin{theorem}{section}
\numberwithin{equation}{theorem}

\theoremstyle{definition}
\newtheorem{definition}[theorem]{Definition}
\newtheorem{example}[theorem]{Example}
\newtheorem{remark}[theorem]{Remark}
\newtheorem{question}[theorem]{Question}
\newtheorem*{question*}{Question}

\newcommand{\im}{\text{\upshape im}}
\DeclareMathOperator{\divv}{div}

\newcommand{\lan}{\langle}
\newcommand{\ran}{\rangle}

\DeclareMathOperator{\Hom}{Hom}

\DeclareMathOperator{\Kdim}{Kdim}

\newcommand\kk{{\Bbbk}}

\begin{document}

\title[Twists and related properties]
{Twists of graded Poisson algebras\\
and related properties}

\author{Xin Tang, Xingting Wang and James J. Zhang}

\address{Tang: Department of Mathematics \& Computer Science, 
Fayetteville State University, Fayetteville, NC 28301,
USA}

\email{xtang@uncfsu.edu}

\address{Wang: Department of Mathematics, Howard University, Washington, 
DC, 20059, USA} 

\email{xingting.wang@howard.edu}

\address{Zhang: Department of Mathematics, Box 354350,
University of Washington, Seattle, Washington 98195, USA}

\email{zhang@math.washington.edu}

\begin{abstract}
We introduce a Poisson version of the graded twist of a 
graded associative algebra and prove that every graded 
Poisson structure on a connected graded polynomial ring 
$A:=\Bbbk[x_1,\ldots,x_n]$ is a graded twist of a
unimodular Poisson structure on $A$, namely, if $\pi$ is 
a graded Poisson structure on $A$, then $\pi$ has a 
decomposition
$$\pi~=~\pi_{unim} +\frac{1}{\sum_{i=1}^n \deg x_i} E\wedge 
{\mathbf m}$$
where $E$ is the Euler derivation, $\pi_{unim}$ is the 
unimodular graded Poisson structure on $A$ corresponding to 
$\pi$, and ${\mathbf m}$ is the modular derivation of 
$(A,\pi)$. This result is a generalization of the same result 
in the quadratic setting. The rigidity of graded twisting, 
$PH^1$-minimality, and $H$-ozoneness are studied. As an 
application, we compute the Poisson cohomologies of the
quadratic Poisson structures on the polynomial ring of three 
variables when the potential is irreducible, but not 
necessarily having isolated singularities.
\end{abstract}




\maketitle


\section*{Introduction}
\label{xxsec0}

Poisson algebras have recently been studied extensively by 
many researchers, see e.g., \cite{Ba2, Ba3, Go1, Go2, GLa, 
GLe, JO, LS, LuWW1, LuWW2, LvWZ1}, with topics related to 
(twisted) Poincar{\' e} duality and the modular derivation, 
Poisson Dixmier-Moeglin equivalences, Poisson enveloping 
algebras and so on. Poisson algebras have been used in the 
representation theory of PI Sklyanin algebras \cite{WWY1, WWY2}. 
Isomorphism problem and cancellation problem in the Poisson 
setting have been investigated in \cite{GW1, GW2}.

Let $\Bbbk$ be a base field. Except for Sections \ref{xxsec1}
and \ref{xxsec2}
we further assume that $\Bbbk$ is of characteristic zero.
Quadratic Poisson structures on $\Bbbk[x_1,\ldots,x_n]$ with
$\deg(x_i)=1$ for all $i=1,\cdots,n$ have been playing an 
important role in several other subjects, see papers \cite{LX} 
by Liu-Xu, \cite{Bo} by Bondal, and \cite{Py} by Pym. 
Deformation quantizations of such Poisson structures are 
homogeneous coordinate rings of quantum ${\mathbb P}^{n-1}$s. 
In general, such a deformation quantization is skew-Calabi-Yau; 
while it is Calabi-Yau if and only if the Poisson structure on 
$\Bbbk[x_1,\ldots,x_n]$ is unimodular \cite{Do}. 

In additional to the quadratic case, we are interested in 
weighted Poisson structures on $\Bbbk[x_1,\ldots,x_n]$ where 
$\deg x_i>0$ for all $i=1,\ldots,n$. Note that deformation 
quantizations of weighted Poisson structures are homogeneous 
coordinate rings of weighted quantum ${\mathbb P}^{n-1}$s. If 
$\pi$ is a graded Poisson structure on $\Bbbk[x_1,\ldots,x_n]$ 
where $\sum_{i=1} \deg x_i\neq 0$ in the base field $\Bbbk$, 
we prove that $\pi$ has a decomposition
\begin{equation}
\label{E0.0.1}\tag{E0.0.1}
\pi~=~\pi_{unim} +\frac{1}{\sum_{i=1}^n \deg x_i} E\wedge 
{\mathbf m}
\end{equation}
where $E$ is the Euler derivation, $\pi_{unim}$ is the 
unimodular graded Poisson structure on $\Bbbk[x_1,\ldots,x_n]$ 
corresponding to $\pi$, and ${\mathbf m}$ is the modular 
derivation of $(\Bbbk[x_1,\ldots,x_n],\pi)$. If $\deg x_i=1$ 
for all $i$, \eqref{E0.0.1} was observed by Bondal \cite{Bo}, 
Liu-Xu \cite{LX}, and in the book \cite[Theorem 8.26]{LPV}. 
Similar to the ideas in \cite{Py}, to classify all graded 
Poisson structures on polynomial rings where $\deg x_i>0$, it 
is a good idea to first classify unimodular ones.

To prove \eqref{E0.0.1}, we will use a Poisson version of the 
graded twist \cite{Zh}. Let $A$ be a ${\mathbb Z}$-graded 
Poisson algebra such that both the commutative multiplication 
$\cdot$ and the Poisson bracket $\pi:=\{-,-\}$ are 
graded of degree 0. If $a\in A$ is homogeneous, we use $|a|$ 
to denote its degree. Define the Euler derivation $E$ of $A$ by
$$E(a)~=~|a|\, a$$
for all homogeneous element $a\in A$. Let $\delta$ be a graded 
Poisson derivation of $A$. We define a new Poisson structure, 
denoted by $\pi_{new}:=\{-,-\}_{new}$, to be
\begin{equation}
\label{E0.0.2}\tag{E0.0.2}
\{a,b\}_{new}~:=~\{a,b\}+E(a) \delta(b)-\delta(a) E(b)
\end{equation}
for all homogeneous elements $a,b\in A$, or equivalently
$$\pi_{new}~:=~\pi+E\wedge \delta.$$
We will show that $(A,\cdot, \{-,-\}_{new})$ (or $(A,\pi_{new})$)
is a graded Poisson algebra in Section 2 and it is denoted by 
$A^{\delta}$.

Now we state some results. Let $A$ be a polynomial algebra 
$\Bbbk[x_1,\ldots,x_n]$ and let $\delta$ be a derivation of $A$. 
By \cite[(4.21)]{LPV}, the {\it divergence} of $\delta$ is 
defined to be
\begin{equation}
\label{E0.0.3}\tag{E0.0.3}
\divv (\delta) ~:=~\sum_{i=1}^n 
\frac{\partial \delta(x_i)}{\partial x_i},
\end{equation}
which is independent of the choices of generators
$\{x_1,\ldots,x_n\}$ [Definition \ref{xxdef1.1} and
Lemma \ref{xxlem1.2}]. For a more general Poisson algebra,
the definition of $\divv(\delta)$ will be given in Definition 
\ref{xxdef1.1} which is dependent on the volume form. Recall 
a {\it Hamiltonian derivation} of a Poisson algebra $A$ is 
given by $H_a:=\{a,-\}$ for any $a\in A$. The 
{\it modular derivation} of $A$ is defined by 
\begin{equation}
\label{E0.0.4}\tag{E0.0.4}
{\mathbf m}(a)~:=~-\divv (H_a)
\end{equation}
for all $a\in A$ [Definition \ref{xxdef1.4}]. We need the 
following lemma that concerns the divergence of the modular 
derivation. 

\begin{lemma} 
\cite[Corollary 3.10]{Wa} \cite[Proposition 4.17]{LPV}
\label{xxlem0.1}
Let $A$ be a Poisson 
algebra with volume form $\nu$ and ${\mathbf m}$ be the modular 
derivation of $A$ corresponding to
$\nu$. Then $\divv({\mathbf m})=0$.
\end{lemma}

\begin{proof} Following the notation of \cite[Theorem 3.5]{Wa}, we denote 
${\mathbf m}$ by $\phi$ and $\nu$ by ${\text{vol}}$. By the proof of 
\cite[Corollary 3.10]{Wa}, ${\mathcal L}_{{\mathbf m}}(\nu)=0$. 
Then, by Definition \ref{xxdef1.1}, $\divv({\mathbf m})=0$.
\end{proof}


According to the ideas of Dolgushev \cite{Do}, the modular 
derivation of a Poisson algebra is corresponding to the Nakayama 
automorphism of a noetherian AS regular algebra. Hence the above 
lemma is a Poisson version of \cite[Corollary 5.5]{RRZ2} which 
says that the Nakayama automorphism of a noetherian AS regular 
algebra has the homological determinant 1.

When $A$ is a polynomial algebra $\Bbbk[x_1,\ldots,x_n]$ with 
any Poisson structure, the definitions of the divergence $\divv$ 
and the modular derivation $\mathbf m$ are independent of choices 
of the volume form. Here is one of main results of this paper.

\begin{theorem}
\label{xxthm0.2}
Let $\delta$ be a graded Poisson derivation of a 
${\mathbb Z}$-graded Poisson polynomial algebra 
$A:=\Bbbk[x_1,\ldots,x_n]$. Let ${\mathbf n}$ be the modular 
derivation of $A^{\delta}$. Then 
$${\mathbf n}={\mathbf m}+(\sum_{i=1}^n\deg x_i)\delta 
-\divv (\delta) E.$$
\end{theorem}

Note that Theorem \ref{xxthm0.2} holds even when 
${\text{char}}\; \Bbbk>0$, see Remark \ref{xxrem3.7}.

If we consider the analogy between the modular derivation of 
a Poisson algebra and the Nakayama automorphism of a graded 
skew Calabi-Yau algebra \cite{Do}, Theorem \ref{xxthm0.2} is a 
Poisson version of \cite[Theorem 0.3]{RRZ1}. Combining Theorem 
\ref{xxthm0.2} with Lemma \ref{xxlem0.1}, we obtain

\begin{corollary}
\label{xxcor0.3}
Let $A$ is a weighted 
graded Poisson algebra $\Bbbk[x_1,\ldots,x_n]$ with $\deg x_i>0$ 
for all $i$. Let $\delta=-{\frac{1}{\mathfrak l}} {\mathbf m}$ 
where ${\mathfrak l}=\sum_{i=1}^n \deg (x_i)$. Then $A^{\delta}$ 
is unimodular. As a consequence, \eqref{E0.0.1} holds.
\end{corollary}

Let $A$ be a ${\mathbb Z}$-graded Poisson algebra. Suppose
$\delta$ is a derivation of $A$ and $a,b,c$ are homogeneous
elements of $A$. Let
\begin{align}\notag
p(\{-,-\}, \delta; a,b,c):=&
|a|a[\delta(\{b,c\})-\{\delta(b),c\}
-\{b, \delta(c)\}]\\
&\label{E0.3.1}\tag{E0.3.1}
-|b|b[\delta(\{a,c\})-\{\delta(a),c\}
-\{a, \delta(c)\}]\\
&+|c|c[\delta(\{a,b\})-\{\delta(a),b\}
-\{a, \delta(b)\}].
\notag
\end{align}

\begin{definition}
\label{xxdef0.4}
Let $A$ be a ${\mathbb Z}$-graded Poisson algebra. A derivation
$\delta$ of $A$ is called {\it semi-Poisson} if 
$p(\{-,-\}, \delta, a,b,c)=0$ for all homogeneous
elements $a,b,c$ in $A$.
\end{definition}

It is clear that
$${\text{Poisson derivation}}
\Rightarrow {\text{ semi-Poisson derivation}}
\Rightarrow {\text{ derivation}}$$
and opposite implications are not true [Example \ref{xxex2.6}]. 
Let $Gspd(A)$ (resp. $Gpd(A)$) be the set of graded 
semi-Poisson derivations (resp. graded Poisson derivations) of 
degree 0. We prove the following 

\begin{theorem}
\label{xxthm0.5}
Let $A$ be a graded Poisson algebra $\Bbbk[x_1,
\ldots,x_n]$ with $\deg x_i>0$ for all $i$.
\begin{enumerate}
\item[(1)]
If $A$ is unimodular, then $Gspd(A)=Gpd(A)$. 
\item[(2)]
If $B$ is a twist of $A$, then 
$Gspd(A)=Gspd(B)$.
\item[(3)]
$Gspd(A)$ is a finite-dimensional Lie algebra.
\end{enumerate}
\end{theorem}

Now we introduce the {\it rigidity of graded twisting} of $A$, 
denoted by $rgt(A)$ (see Definition \ref{xxdef4.3}), to measure 
the complexity/rigidity of a Poisson structure on $A$. We relate
the rigidity with other properties. We say a Poisson derivation 
$\phi$ of $A$ is {\it ozone} if $\phi(z)=0$ for all $z$ in 
the Poisson center of $A$. It is obvious that every Hamiltonian 
derivation is ozone, but the converse is not true 
in general. Recall that $E$ denotes the Euler derivation.

Let $M$ be a ${\mathbb Z}$-graded $\Bbbk$-vector space. 
The Hilbert series of $M$ is defined to be
\begin{equation}
\label{E0.5.1}\tag{E0.5.1}
h_M(t)=\sum_{i\in {\mathbb Z}} (\dim M_i) t^i.
\end{equation}
Let $PH^i(A)$ denote the $i$th Poisson cohomology of 
$A$ \eqref{E1.5.3}. We have the following result.

\begin{theorem}
\label{xxthm0.6}
Let $\Bbbk$ be algebraically closed. 
Let $A$ be the Poisson algebra $\Bbbk[x_1,x_2,x_3]$ with 
$\deg(x_i)=1$ for $i=1,2,3$. Let $Z$ be the Poisson center 
of $A$. Then the following are equivalent.
\begin{enumerate}
\item[(1)]
$rgt(A)=0$.
\item[(2)]
Any graded twist of $A$ is isomorphic to $A$.
\item[(3)]
The Hilbert series of the graded vector space of Poisson
derivations of $A$ is $\frac{1}{(1-t)^3}$.
\item[(4)]
$h_{PH^1(A)}(t)$ is $\frac{1}{1-t^3}$.
\item[(5)]
$h_{PH^1(A)}(t)$ is equal to $h_Z(t)$.
\item[(6)]
Every Poisson derivation $\phi$ has a decomposition
$$\phi=zE+H_a$$
where $z\in Z$ and $a\in A$. Here $z$ is unique and 
$a$ is unique up to a central element. 
\item[(7)]
Every ozone derivation is Hamiltonian. 
\item[(8)]
$A$ is unimodular and the potential is irreducible.
\item[(9)]
$h_{PH^3(A)}(t)-h_{PH^2(A)}(t)=t^{-3}$.
\end{enumerate}
\end{theorem}

Some partial generalizations of the above theorem to the
higher dimensional case are given in Section \ref{xxsec7}.
As an application, we have the following result.

\begin{corollary}
\label{xxcor0.7}
Let $\Bbbk$ be algebraically closed. Let $A$ be the unimodular 
quadratic Poisson structure on $\Bbbk[x,y,z]$ with irreducible 
potential $\Omega$. Then 
\begin{enumerate}
\item[(1)]
$\displaystyle h_{PH^0(A)}(t)=\frac{1}{1-t^3}$.
\item[(2)]
$\displaystyle h_{PH^1(A)}(t)=\frac{1}{1-t^3}$.
\item[(3)]
$\displaystyle h_{PH^2(A)}(t)=\frac{1}{t^3}(\frac{(1+t)^3}{1-t^3}-1)$.
\item[(4)]
$\displaystyle h_{PH^3(A)}(t)=\frac{(1+t)^3}{t^3(1-t^3)}$.
\end{enumerate}
\end{corollary}

When the potential $\Omega$ has isolated singularities, the 
Poisson cohomologies have been computed by several authors, 
see \cite{Pe1,Pe2,Pi1, Pi2, VdB} and the references therein. 
The above corollary is probably the first computation of the 
Poisson cohomologies when $\Omega$ is irreducible, but does 
not have isolated singularities.

The paper is organized as follows. Section 1 recalls some 
basic definitions such as divergence and modular derivation.
In Section 2 we introduce the Poisson version of a graded 
twist. The proofs of Theorem \ref{xxthm0.2} and Corollary 
\ref{xxcor0.3} are given in Section 3. The rigidity of 
graded twisting is introduced in Section 4 and Theorem 
\ref{xxthm0.5} is proven there. In Sections 5 and 6 we 
compute the rigidity of some Poisson structures on 
polynomial rings. Theorem \ref{xxthm0.6} and Corollary 
\ref{xxcor0.7} are proved in Section 7.

\section{Preliminaries}
\label{xxsec1}

In this section we recall several definitions such as 
{\it divergence}, {\it modular derivation}, and 
{\it Poisson cohomology}. Other 
basic definitions about Poisson algebras can be found in the 
book \cite{LPV}. Everything in this section is well known.

In Sections \ref{xxsec1} and \ref{xxsec2}, let $\Bbbk$ be 
a base field of any characteristic. 
Let $\Omega^1(A)$ be the module of K{\" a}hler differentials 
over $A$ \cite[Sect. 3.2.1]{LPV}. For each $k\geq 0$, let
$\Omega^k(A)$ be $\wedge^p_A \Omega^1(A)$ 
\cite[Sect. 3.2.2]{LPV}. Let $d=\Kdim A$ where $\Kdim$ denotes
the Krull dimension. If $A$ is smooth and $\Omega^d(A)$ is 
a free $A$-module with a generator $\nu$, then $\nu$ is 
called a {\it volume form} of $A$. The differential 
$d:A\to \Omega^1(A)$ extends to a well-defined differential of 
the complex $\Omega^{\bullet}(A)$ and the complex 
$(\Omega^{\bullet},d)$ is called the algebraic {\it de Rham 
complex} of $A$.

For each $k\geq 0$, let ${\mathfrak X}^k(A)$ be the set of 
skew-symmetric $k$-derivations of $A$. It is also true that
\begin{equation}
\label{E1.0.1}\tag{E1.0.1}
{\mathfrak X}^p(A)\cong \Hom_{A}(\Omega^p(A),A)
\end{equation}
for all $p\geq 0$ \cite[(3.15)]{LPV}.

For every element $P\in {\mathfrak X}^p(A)$, the {\it internal 
product} with respect to $P$, denoted by $\iota_P$, is an 
$A$-module map 
$$\iota_P: \Omega^{\bullet}(A)\to \Omega^{\bullet-p}(A)$$
which is determined by
\begin{equation}
\label{E1.0.2}\tag{E1.0.2}
\iota_P(dF_1\wedge dF_2 \wedge \cdots \wedge dF_k)=
\begin{cases} 0& k<p,\\
\sum_{\sigma\in {\mathbb S}_{p,k-p}} sgn(\sigma)
P[F_{\sigma(1)},\ldots, F_{\sigma(p)}] \\
\qquad\qquad dF_{\sigma(p+1)} \wedge \cdots \wedge dF_{\sigma(k)}
\in \Omega^{k-p}(A) &k\geq p
\end{cases}
\end{equation}
for all $dF_1\wedge dF_1 \wedge \cdots \wedge dF_k \in 
\Omega^k(A)$. Here ${\mathbb S}_{p,q}\subset {\mathbb S}_k$ 
is the set of $(p,q)$-shuffles with $p+q=k$. 

For every $P\in {\mathfrak X}^p(A)$, the {\it Lie derivative} 
with respective to $P$ is defined to be
\begin{equation}
\label{E1.0.3}\tag{E1.0.3}
{\mathcal L}_{P}=[\iota_P,d]: \Omega^{\bullet}(A)
\to \Omega^{\bullet-p+1}(A),
\end{equation}
see \cite[(3.49)]{LPV}. Below is the definition 
of the divergence of a derivation. In several definitions in 
this paper we assume that $A$ is a smooth Poisson algebra with 
a fixed volume form $\nu$. 

\begin{definition} \cite[(4.20)]{LPV}
\label{xxdef1.1} Let $\delta$ be a derivation of $A$, namely, 
$\delta \in {\mathfrak X}^1(A)$. The {\it divergence} of 
$\delta$, denoted by $\divv (\delta)$, is an element in $A$ 
defined by the equation
\begin{equation}
\label{E1.1.1}\tag{E1.1.1}
{\mathcal L}_{\delta} (\nu)= \divv (\delta) \nu.
\end{equation}
\end{definition}

It is clear that the divergence of $\delta$ is dependent on the 
volume form $\nu$, but independent of the Poisson structure of 
$A$. The definition of the divergence of a skew-symmetric 
$k$-derivation, for $k\geq 2$, can be found in 
\cite[Sect. 4.3.3]{LPV}. 

Part (1) of the following lemma justifies the definition of 
the divergence given in \eqref{E0.0.3}. Let $G$ be an abelian 
group (or semigroup). A $G$-graded algebra with a Poisson 
structure is called a $G$-graded Poisson structure if 
$\deg(\{a,b\})=\deg a+\deg b$ for all homogeneous elements 
$a,b\in A$.

\begin{lemma}
\label{xxlem1.2}
Let $A$ be a Poisson polynomial algebra $\Bbbk[x_1,\ldots,x_n]$.
\begin{enumerate}
\item[(1)]
\begin{equation}
\label{E1.2.1}\tag{E1.2.1}
\divv (\delta)=\sum_{i=1}^n 
\frac{\partial \delta(x_i)}{\partial x_i}.
\end{equation}
\item[(2)]
If $A$ is a ${\mathbb Z}$-graded Poisson polynomial algebra with 
$x_i$ homogeneous for all $i$ and if $\delta$ is graded (of degree 0), then 
$\divv (\delta)\in A_0$.
\item[(3)]
If $A$ is a connected ${\mathbb N}$-graded Poisson polynomial 
algebra with $\deg x_i>0$ for all $i$ and if $\delta$ is graded
(of degree 0), 
then $\divv (\delta)\in \Bbbk$.
\item[(4)]
Suppose, in addition to {\rm{(3)}}, $\deg x_i=1$ for all $i$.
Let $\delta$ be a derivation of $A$ (of degree 0). Write 
$$\delta(x_i)=\sum_{j=1}^n c_{ij} x_j$$
where $c_{ij}\in \Bbbk$ for all $1\leq i,j\leq n$.
Then 
\begin{equation}
\label{E1.2.2}\tag{E1.2.2}\divv (\delta)=\sum_{i=1}^n c_{ii}.
\end{equation} 
\item[(5)]
Let $A$ be a ${\mathbb Z}$-graded Poisson algebra with $\deg(x_i)
\in {\mathbb Z}$ and let $E$ be the Euler derivation of $A$ defined 
in the introduction. Then $\divv(E)=\deg (\nu)$.
\end{enumerate}
\end{lemma}

\begin{proof} (1) Since $\Bbbk[x_1,\ldots,x_n]$, $\nu:=dx_1 
\wedge \cdots \wedge dx_n$ is a volume form. By the definition 
of the Lie derivative ${\mathcal L}_{\delta}$,
$$\begin{aligned}
{\mathcal L}_{\delta} \nu
&=d (\sum_{i=1} (-1)^{i-1} \delta(x_i) 
dx_1 \wedge \cdots \wedge\widehat{dx_i} \wedge\cdots \wedge dx_n)\\
&=\sum_{i=1}^n (-1)^{i-1} 
(\sum_{j=1}^n \frac{ \partial \delta(x_i)}{\partial x_j}
dx_j) \wedge dx_1 \wedge \cdots \wedge\widehat{dx_i} 
\wedge\cdots \wedge dx_n\\
&=(\sum_{i=1}^n \frac{ \partial \delta(x_i)}{\partial x_i})
dx_1 \wedge \cdots \wedge dx_n\\
&=(\sum_{i=1}^n \frac{ \partial \delta(x_i)}{\partial x_i})
\nu.
\end{aligned}
$$
Then the assertion follows.

(2) Since $\deg \delta=0$, $\deg \delta (x_i)=\deg x_i$. As a 
consequence, $\deg (\frac{ \partial \delta(x_i)}{\partial x_i})=0$.
By part (1), $\deg (\divv \delta)=0$. The assertion follows.

(3) This follows from part (3) and the fact that $A_0=\Bbbk$.

(4) This follows from part (1) and the fact that 
$\frac{ \partial \delta(x_i)}{\partial x_i}=c_{ii}$ for all 
$i$.

(5) In this case, $\nu =dx_1\wedge \cdots \wedge dx_n$ and $E(x_i)
=(\deg x_i) x_i$. The assertion follows from \eqref{E1.2.1}.
\end{proof}

We recall the following definition.

\begin{definition} 
\label{xxdef1.4}
Let $A$ be a Poisson algebra with volume form $\nu$. 
\begin{enumerate}
\item[(1)]
\cite[Definition 4.10]{LPV}
The {\it modular derivation} (or {\it modular vector field}) of 
$A$ associated to $\nu$ is defined to be
$${\mathbf m}(a):=-\divv H_a$$
for all $a\in A$, or equivalently,
$${\mathcal L}_{H_a}(\nu)=-{\mathbf m}(a) \nu.$$
\item[(2)]
\cite[Definition 4.12]{LPV}
If ${\mathbf m}=0$ for some volume form $\nu$, then 
$A$ is called {\it unimodular}.
\end{enumerate}
\end{definition}

If $A=\Bbbk[x_1,\ldots,x_n]$, then ${\mathbf m}$ is 
independent of the choice of the volume forms $\nu$.

Let us give an easy example.

\begin{example}
\label{xxex1.5} 
Let $A$ be the Poisson polynomial algebra $\Bbbk[x_1,x_2]$ 
with $\{x_1,x_2\}=x_1^n$ for some integer $n\geq 0$. It is 
easy to check that ${\mathbf m}(x_1)=0$ and that 
${\mathbf m}(x_2)=nx_1^{n-1}$. If ${\rm{char}}\; \Bbbk=p>0$ 
and $p\mid n$, then $A$ is unimodular. This is an interesting 
fact. Now suppose $n=2$. Since $\{x_1x_2,x_2\}=x_1^2x_2$, 
${\mathbf m}(\{x_1x_2,x_2\})={\mathbf m}(x_1^2x_2)=2x_1^3$. 
As a consequence,  $\divv ([\delta_1, \delta_2])$ is in 
general nonzero for any two derivations $\delta_1,\delta_2$.
\end{example}

Next we review the Poisson cohomology. Let $(A,\pi)$ be a 
Poisson algebra. For each $k$, ${\mathfrak X}^k(A)$ is the 
space of skew-symmetric $k$-derivations of $A$. The Poisson 
coboundary map $d_{\pi}: {\mathfrak X}^{\bullet}(A)\to 
{\mathfrak X}^{\bullet+1}(A)$ is defined as follows. For 
any $Q\in {\mathfrak X}^q(A)$, where $q\in {\mathbb N}$, we define
\begin{align}
d_{\pi}^q(Q)[F_0,\ldots,F_q]
:=&\sum_{i=0}^q (-1)^i \{ F_i, Q[F_0,\ldots,\widehat{F_i},\ldots,F_q]\}
\label{E1.5.1}\tag{E1.5.1}\\
&
+\sum_{0\leq i\leq j\leq q}
(-1)^{i+j}
Q[\{F_i,F_j], F_0,\ldots,\widehat{F_i},\ldots,\widehat{F_j},\ldots,F_q]
,\notag
\end{align}
for all $F_0,\cdots,F_q\in A$. In particular,
$$\begin{aligned}
d_{\pi}^0(Q)[F_0]&=\{F_0,Q\},\\
d_{\pi}^1(Q)[F_0,F_1]&=
\{F_0,Q[F_1]\}-\{F_1, Q[F_0]\}-Q[\{F_0,F_1\}],\\
d_{\pi}^2(Q)[F_0,F_1,F_2]&
=\{F_0, Q[F_1,F_2]\}-\{F_1,Q[F_0,F_2]\}+\{F_2, Q[F_0,F_1]\}\\
&\quad 
-Q[\{F_0,F_1\},F_2]+Q[\{F_0,F_2\},F_1]-Q[\{F_1,F_2\},F_0]
\end{aligned}
$$

For $P\in {\mathfrak X}^{p}(A)$ and $Q\in {\mathfrak X}^q(A)$, the wedge
product $P\wedge Q\in {\mathfrak X}^{p+q}(A)$ is the skew-symmetric 
$(p+q)$-derivation of $A$, defined by
$$\begin{aligned}
(P\wedge Q)&[F_1,\cdots,F_{p+q}]:= \\
&\sum_{\sigma\in {\mathbb S}_{p,q}} sgn(\sigma) 
P[F_{\sigma(1)},\ldots, F_{\sigma(p)}]\,
Q[F_{\sigma(p+1)},\ldots,F_{\sigma(p+q)}],
\end{aligned}$$
for all $F_1,\cdots,F_{p+q}\in A$. In particular, if 
$P\in {\mathfrak X}^{1}(A)$ and 
$Q\in {\mathfrak X}^2(A)$, then we have
\begin{equation}
\label{E1.5.2}\tag{E1.5.2}
(P\wedge Q)[F_1,F_2,F_3]=
P[F_1]Q[F_2,F_3]-P[F_2]Q[F_1,F_3]+P[F_3]Q[F_1,F_2].
\end{equation}

Therefore $({\mathfrak X}^\ast(A), \wedge, d)$ is a dga 
(differential graded algebra). For each $q\geq 0$, the 
{\it $q$-th Poisson cohomology} of $A$ is defined to be
\begin{equation}
\label{E1.5.3}\tag{E1.5.3}
PH^q(A):=\frac{\ker d_{\pi}^q}{\im d^{q-1}_{\pi}}.
\end{equation}
It is clear from the definition that the 1st Poisson cohomology 
of $A$ is 
\begin{equation}
\label{E1.5.4}\tag{E1.5.4}
PH^1(A):=\frac{{\text{the set of Poisson derivations}}}
{{\text{the set of Hamiltonian derivations}}}.
\end{equation}

If $A$ is a quadratic Poisson algebra $\Bbbk[x,y,z]$ (with 
$\deg(x)=\deg(y)=\deg(z)=1$), then the complex
$({\mathfrak X}^{\bullet}(A),d_{\pi})$ is \cite[(15)]{Pi1}
$$0\to A\to (A[1])^{\oplus 3} \to (A[2])^{\oplus 3}
\to A[3]\to 0.$$
By the additivity of Hilbert series, we have
\begin{equation}
\label{E1.5.5}\tag{E1.5.5}
\sum_{i=0}^3 (-1)^{i} h_{PH^i(A)}(t)=-t^{-3}.
\end{equation}
It is easy to check that 
\begin{enumerate}
\item[(a)]
the lowest degree of nonzero elements in $PH^0(A)$ is 0 
and $PH^0(A)_0=\Bbbk$.
\item[(b)]
the lowest degree of nonzero elements in $PH^1(A)$ is $\geq -1$.
\item[(c)]
the lowest degree of nonzero elements in $PH^2(A)$ is 
$\geq -2$. 
\item[(d)]
the lowest degree of nonzero elements in $PH^3(A)$ is $-3$ 
and $PH^3(A)_{-3}=\Bbbk$.
\end{enumerate}
If $A$ is further unimodular, then
\begin{enumerate}
\item[(e)]
the lowest degree of nonzero elements in $PH^2(A)$ is $-2$ 
and $PH^2(A)_{-2}=\Bbbk^{\oplus 3}$.
\end{enumerate}

A natural operation on ${\mathfrak X}^{\bullet}(A)$ is the 
Schouten bracket
$$[\cdot, \cdot]_{S}: {\mathfrak X}^{p}(A)
\times {\mathfrak X}^{q}(A)\to {\mathfrak X}^{p+q-1}(A)$$
for all $p,q\geq 0$. We refer to \cite[Section 3.3.2]{LPV}
for the precise definition. By \cite[(4.5)]{LPV},
$$d_{\pi}(\cdot) =-[\cdot ,\pi]_{S}.$$
By \cite[Proposition 3.7]{LPV}, $({\mathfrak X}^{\bullet}(A),
\wedge, [\cdot,\cdot]_S)$ is a Gerstenhaber algebra .

Let $A$ or $(A, \pi)$ be a Poisson algebra with Poisson bracket
$\pi$. Let $\xi$ be any nonzero scalar. We define a new Poisson
bracket $\pi_{\xi}:=\xi \pi$ or $\{-,-\}_{\xi}:=\xi\{-,-\}$.
Then it is easy to see that $A':=(A, \pi_{\xi})$ is indeed
a Poisson algebra. In general, $A'$ is not isomorphic to 
$A$, but they are closely related as follows.

\begin{lemma}
\label{xxlem1.6}
Retain the notations as above.
Let $d_{\pi}^q$ {\rm{(}}resp. $d_{\pi'}^q${\rm{)}} be the 
differential of ${\mathfrak X}^{\bullet}(A)$ {\rm{(}}resp. 
${\mathfrak X}^{\bullet}(A')${\rm{)}} as defined in \eqref{E1.5.1}.
\begin{enumerate}
\item[(1)]
$d_{\pi'}^q=\xi d_{\pi}^q$ for all $q$.
\item[(2)]
$\ker d_{\pi'}^q=\ker d_{\pi}^q$ for all $q$.
\item[(3)]
$\im d_{\pi'}^q=\im d_{\pi}^q$ for all $q$.
\item[(4)]
$PH^q(A)=PH^q(A')$ for all $q$.
\end{enumerate}
\end{lemma}

\section{Twists of graded Poisson algebras}
\label{xxsec2}
Let $G$ be an abelian group and $A$ be a $G$-graded Poisson 
algebra (namely, both the multiplication $\cdot$ and the Poisson bracket $\{-,-\}$ of $A$ are 
graded of degree 0). We use $g$ for elements in $G$. If 
$a$ is a homogeneous element in $A$, we use $|a|$ to denote 
its degree in $G$.

The aim of this section is to define a Poisson version of 
the graded twist of graded associative algebras \cite{Zh}.

\begin{definition}
\label{xxdef2.1} 
Let $\delta:=\{\delta_{g}\mid g\in G\}$ be a set of graded 
derivations of $A$ (of degree 0). We say $\delta$ is a 
{\it Poisson twisting system} if it satisfies the following 
conditions 
\begin{enumerate}
\item[(1)]
For all $g,h\in G$,
\begin{equation}
\label{E2.1.1}\tag{E2.1.1}
\delta_{g} \delta_{h}~=~\delta_{h}\delta_{g}.
\end{equation} 
\item[(2)]
For homogeneous elements $a,b\in A$,
\begin{equation}
\label{E2.1.2}\tag{E2.1.2}
\delta_{|ab|}~ =~\delta_{|a|}+\delta_{|b|}.
\end{equation}
\item[(3)]
For homogeneous elements $a,b,c,\in A$,
\begin{align}
\label{E2.1.3}\tag{E2.1.3}
 a[\delta_{|a|}(\{b,c\})-\{\delta_{|a|}(b),c\}
-\{b, \delta_{|a|}(c)\}]
-b&[\delta_{|b|}(\{a,c\})-\{\delta_{|b|}(a),c\}
-\{a, \delta_{|b|}(c)\}]\\
+c[\delta_{|c|}(\{a,b\})-\{\delta_{|c|}(a),b\}
-\{a, \delta_{|c|}(b)\}]&=0.
\notag
\end{align}
\end{enumerate}
\end{definition}

\begin{remark}
\label{xxrem2.2} 
\begin{enumerate}
\item[(1)]
The definition of a Poisson twisting system is a ``translation'' 
of the twisting system in the setting of graded associative 
algebras given in \cite[Definition 2.1]{Zh}.

\item[(2)]
Let 
\begin{align}
\label{E2.2.1}\tag{E2.2.1}
p(\{-,-\},\delta\; a,b,c)~:=~
&a[\delta_{|a|}(\{b,c\})-\{\delta_{|a|}(b),c\}
-\{b, \delta_{|a|}(c)\}]\\
\notag
&-b[\delta_{|b|}(\{a,c\})-\{\delta_{|b|}(a),c\}
-\{a, \delta_{|b|}(c)\}]\\
\notag
&+c[\delta_{|c|}(\{a,b\})-\{\delta_{|c|}(a),b\}
-\{a, \delta_{|c|}(b)\}].
\end{align}
Then \eqref{E2.1.3} is equivalent to $p(\{-,-\};\delta;a,b,c)=0$.
If each $\delta_g$ is a Poisson derivation, it is 
automatic that $p(\{-,-\},\delta,;a,b,c)=0$. The converse 
is not true, see Example \ref{xxex2.6}. 
\item[(3)]
Suppose $G={\mathbb Z}$ and let $\phi=\delta_1$. By \eqref{E2.1.2}, 
$\delta_n=n\phi$. It is clear that 
$$p(\{-,-\},\delta\; a,b,c)=(E\wedge d_{\pi}^1 (\phi))(a,b,c),$$
which implies that \eqref{E2.1.3} is equivalent to 
$E\wedge d_{\pi}^1(\phi)=0$.
By \cite[Sect. 4.3]{LPV} and the fact that $d_{\pi}^1(E)=0$, 
the equation $E\wedge d_{\pi}^1(\phi)=0$ is equivalent to 
$d_{\pi}^2 (E\wedge \phi)=0$. Assume that $a,b,c$ are homogeneous
of degree one. Then 
\begin{align}
\label{E2.2.2}\tag{E2.2.2}
p(\{-,-\},\delta; a,b,c)~=~
&|a|a[\phi(\{b,c\})-\{\phi(b),c\}
-\{b, \phi(c)\}]\\
\notag
&-|b|b[\phi(\{a,c\})-\{\phi(a),c\}
-\{a, \phi(c)\}]\\
\notag
&+|c|c[\phi(\{a,b\})-\{\phi(a),b\}
-\{a, \phi(b)\}]
\end{align}
which agrees with $p(\{-,-\}, \phi; a,b,c)$ as defined in 
\eqref{E0.3.1}.
\item[(4)]
Let $G={\mathbb Z}$ and $A$ be a ${\mathbb Z}$-graded Poisson 
algebra. A convenient Poisson twisting system is constructed as 
follows. Let $\phi$ be a graded Poisson derivation of $A$ (namely,
$d_{\pi}^1(\phi)=0$). For 
each $n\in {\mathbb Z}$, let $\delta_n:=n \phi$ and 
$\delta:=\{\delta_{n} \mid n\in {\mathbb Z}\}$. Then 
\eqref{E2.1.1} and \eqref{E2.1.2} are obvious and \eqref{E2.1.3} 
follows from the fact that $\delta_{n}$ is a Poisson derivation, 
see part (2) or (3).
\end{enumerate}
\end{remark}

\begin{example}
\label{xxex2.3}
Let $G={\mathbb Z}/(n)$ for some positive integer $n$. Let 
$A$ be a $G$-graded Poisson algebra and $\delta$ be a graded 
Poisson derivation of $A$. Suppose $p:={\text{char}}\; \Bbbk$ 
is positive. If $p\mid n$, let $\delta_i=i \delta$ for all 
$i\in G$. Then $\{\delta_i\mid i\in G\}$ is a Poisson twisting 
system. If $p\nmid n$, then there is no nontrivial Poisson 
twisting system for $A$.
\end{example}

Let $A$ be a $G$-graded Poisson algebra and let $\delta
:=\{\delta_g\mid g\in G\}$ be a system of derivations of $A$. 
We define 
\begin{equation}
\label{E2.3.1}\tag{E2.3.1}
\langle a, b \rangle :=\{a,b\}+ a \delta_{|a|}(b)-
b\delta_{|b|}(a)
\end{equation}
for all homogeneous elements $a,b\in A$.

\begin{theorem}
\label{xxthm2.4}
Let $\delta:=\{\delta_{g}\mid g\in G\}$ be a set of 
graded derivations of $A$ satisfying \eqref{E2.1.1} 
and \eqref{E2.1.2}. Then $(A, \langle \cdot, \cdot \rangle)$ 
is a Poisson algebra if and only if \eqref{E2.1.3} holds.
\end{theorem}

\begin{proof}
If $G={\mathbb Z}$, there is a shorter proof using the Schouten 
bracket. For a general abelian group $G$, we make the following 
direct computation.

Claim 1: $\langle \cdot,\cdot\rangle$ 
is skew-symmetric. This claim follows immediately from 
\eqref{E2.3.1}.

Claim 2: for every homogeneous element $a$, $\langle a, -\rangle$ 
is a derivation.

For homogeneous elements $a,b,c$ in $A$, we have 
$$\begin{aligned}
\lan a, bc \ran &=\{a, bc\} + a \delta_{|a|} (bc)-bc \delta_{|bc|}(a)\\
&=\{a, b\}c+\{a,c\}b +a (b\delta_{|a|} (c)+\delta_{|a|}(b)c)-bc \delta_{|bc|}(a)\\
\lan a, b\ran c&=(\{a, b\} +a \delta_{|a|}(b)-b \delta_{|b|}(a))c\\
b \lan a, c\ran&=b(\{a, c\} +a \delta_{|a|} (c)-c \delta_{|c|}(a)).
\end{aligned}
$$

By the above and \eqref{E2.1.2}, we obtain that 
$$\langle a, bc\rangle =\langle a,b\rangle c+b\langle a,c\rangle.$$

Claim 3: $\langle \cdot, \cdot \rangle$ satisfies the Jacobi identity
if and only if \eqref{E2.1.3} holds. As a consequence, 
$(A, \langle \cdot, \cdot \rangle)$ is a Poisson algebra if and only 
if \eqref{E2.1.3} holds.

For homogeneous elements $a,b,c$ in $A$, we have 
$$\begin{aligned}
\langle a, \langle b,c\rangle \rangle
&=\{a, \langle b,c\rangle\}+a \delta_{|a|}(\langle b,c\rangle)
-\langle b,c\rangle \delta_{|bc|}(a)\\
&=\{a, (\{b,c\}+b\delta_{|b|}(c)-c\delta_{|c|}(b))\}\\
&\qquad +a \delta_{|a|} (\{b,c\}+b\delta_{|b|}(c)-c\delta_{|c|}(b))\\
&\qquad\quad -(\{b,c\}+b\delta_{|b|}(c)-c\delta_{|c|}(b))\delta_{|bc|}(a)\\
&=\{a, \{b,c\}\}\\
&\quad +\{a, b\}\delta_{|b|}(c)+b\{a, \delta_{|b|}(c)\}\\
&\quad -\{a, c\}\delta_{|c|}(b)-c\{a, \delta_{|c|}(b)\}\\
&\qquad +a\delta_{|a|}(\{b,c\})\\
&\qquad +a\delta_{|a|}(b)\delta_{|b|}(c)+ab \delta_{|a|}\delta_{|b|}(c)\\
&\qquad -a\delta_{|a|}(c)\delta_{|c|}(b)-ac \delta_{|a|}\delta_{|c|}(b)\\
&\qquad\quad -\{b,c\}\delta_{|bc|}(a)-b\delta_{|b|}(c)\delta_{|bc|}(a)
+c\delta_{|c|}(b)\delta_{|bc|}(a)
\end{aligned}
$$
and
$$\begin{aligned}
\langle \langle a,b\rangle,c\rangle
&=\langle c, \langle b,a \rangle \rangle\\
&=\{c, \{b,a\}\}\\
&\quad +\{c, b\}\delta_{|b|}(a)+b\{c, \delta_{|b|}(a)\}\\
&\quad -\{c, a\}\delta_{|a|}(b)-a\{c, \delta_{|a|}(b)\}\\
&\qquad +c\delta_{|c|}(\{b,a\})\\
&\qquad +c\delta_{|c|}(b)\delta_{|b|}(a)+cb \delta_{|c|}\delta_{|b|}(a)\\
&\qquad -c\delta_{|c|}(a)\delta_{|a|}(b)-ca \delta_{|c|}\delta_{|a|}(b)\\
&\qquad\quad -\{b,a\}\delta_{|ba|}(c)-b\delta_{|b|}(a)\delta_{|ba|}(c)
+a\delta_{|a|}(b)\delta_{|ba|}(c)
\end{aligned}
$$
and
$$\begin{aligned}
\langle b, \langle a,c \rangle \rangle
&=\{b, \{a,c\}\}\\
&\quad +\{b, a\}\delta_{|a|}(c)+a\{b, \delta_{|a|}(c)\}\\
&\quad -\{b, c\}\delta_{|c|}(a)-c\{b, \delta_{|c|}(a)\}\\
&\qquad +b\delta_{|b|}(\{a,c\})\\
&\qquad +b\delta_{|b|}(a)\delta_{|a|}(c)+ba \delta_{|b|}\delta_{|a|}(c)\\
&\qquad -b\delta_{|b|}(c)\delta_{|c|}(a)-bc \delta_{|b|}\delta_{|c|}(a)\\
&\qquad\quad -\{a,c\}\delta_{|ac|}(b)-a\delta_{|a|}(c)\delta_{|ac|}(b)
+c\delta_{|c|}(a)\delta_{|ac|}(b).
\end{aligned}
$$

Using the Jacobi identity
$$-\{a, \{b, c\}\}
+\{\{a, b\} , c\}
+\{b, \{a, c\} \}=0,$$ 
\eqref{E2.1.1}, and \eqref{E2.1.2}, we can simplify 
$$-\lan a, \lan b, c\ran \ran
+\lan \lan c, b\ran , a\ran
+\lan b, \lan a, c\ran \ran$$ 
to 
$$p(\{-,-\},\delta;a,b,c).$$
Therefore $\lan \cdot, \cdot\ran$ satisfies the Jacobi 
identity if and only if $p(\delta;a,b,c)=0$. Claim 3 
follows. The consequence is clear.
\end{proof}

\begin{definition}
\label{xxdef2.5} Let $\delta:=\{\delta_{g}\mid g\in G\}$ be 
a Poisson twisting system of a $G$-graded Poisson algebra $A$. 
Then the new Poisson algebra $(A,\langle -,- \rangle)$ 
given in Theorem \ref{xxthm2.4} is called the 
{\it twist of $A$ by $\delta$} and denoted by $A^{\delta}$.
\end{definition}

\begin{example}
\label{xxex2.6}
Let $A=\Bbbk[x,y]$ be a ${\mathbb Z}$-graded Poisson 
algebra defined by $\{x,y\}=x^2$. Let $\phi$ be the 
derivation sending $x\to -x$ and $y\to y-x$. Let
$\delta_n=n\phi$. It is easy to see that 
$$\begin{aligned}
d_{\pi}^1(\phi)[x,y]:&=
-\phi[\{x,y\}]+\{x,\phi[y]\}+\{\phi[x],y\}\\
&=-\phi[x^2]+\{x,y-x\}+\{-x,y\}=2x^2\neq 0
\end{aligned}
$$ 
which implies that $\phi$ is not a Poisson derivation. 

We claim that $\delta:=\{\delta_n\}$ is a Poisson twisting 
system. Let $f$ be the derivation of $A$ determined by
$$f(x)=0, \quad {\text{and}}\quad f(y)=-x.$$
It is easy to verify that $f$ is a Poisson derivation.
By Remark \ref{xxrem2.2}(4), $f':=\{nf\mid n\in {\mathbb Z}\}$
is a Poisson twisting system, and by Theorem \ref{xxthm2.4}, 
$A^{f'}$ is equipped with a Poisson structure such that
$$\langle x, y\rangle=\{x,y\}+x f(y)-yf(x)=x^2-x^2-0=0.$$
Therefore $A^{f'}$ has trivial Poisson structure. 
Let $g$ be the Poisson derivation of $A^{f'}$ determined by
$$g(x)=-x, \quad {\text{and}}\quad g(y)=y.$$
Let $g'=\{ng\mid n\in {\mathbb Z}\}$. By Remark 
\ref{xxrem2.2}(4), $g'$ is a Poisson twisting system of  $A^{f'}$
and the Poisson structure of $(A^{f'})^{g'}$ is determined by,
for all homogeneous elements $a,b\in A$,
$$\begin{aligned}
\{a,b\}_{new}:&=\langle a, b\rangle +|a|a g(b)-|b|bg(a)\\
&=\{a,b\}+ |a|a f(b)-|b|bf(a)+|a|a g(b)-|b|bg(a)\\
&=\{a,b\}+a h_{|a|}(b)- b h_{|b|}(a)
\end{aligned}
$$
where $h_n=nf+ng$ for all $n\in {\mathbb Z}$. Since 
$f+g$ is a derivation of $A$, by Theorem \ref{xxthm2.4},
$h':=\{h_n\mid n\in {\mathbb Z}\}$ is a Poisson twisting system 
of $A$. It is clear that $\delta=h'$. So $\delta$ is a 
Poisson twisting system. 

Since $\delta$ is a Poisson twisting system, by Remark 
\ref{xxrem2.2}(3), $\phi$ is a graded semi-Poisson derivation. 
By the first paragraph, $\phi$ is not a Poisson derivation.
\end{example}

\begin{lemma}
\label{xxlem2.7}
Suppose $G$ is cyclic. Then the set of Poisson twisting systems 
of $A$ is a $\Bbbk$-vector space.
\end{lemma}

\begin{proof} Let $\delta$ and $\varphi$ be two Poisson twisting 
systems. It is clear that $c \delta$ is a Poisson twisting 
system for all $c\in \Bbbk$. It remains to show 
$h:=\delta+\varphi$ is a Poisson twisting system. 

Since $G$ is cyclic, $h_n=nh_1$. So \eqref{E2.1.1} is clear.
Now \eqref{E2.1.2} and \eqref{E2.1.3} hold as 
these are ``linear'' in terms of $\delta$.
\end{proof}

\begin{remark}
\label{xxrem2.8}
If $G$ is ${\mathbb Z}^2$, then the set of 
Poisson twisting systems of $(\Bbbk[x_1,x_2,x_3],0)$ 
with $\deg x_1=\deg x_2=(1,0)$ and $\deg x_3=(0,1)$ 
is not a $\Bbbk$-vector space. To see this, we consider 
two graded Poisson derivations $\delta_1$ and $\phi_1$
that are not commuting (for example, $\delta_1: 
x_1\to x_1, x_2\to 0, x_3\to 0$
and $\phi_1: x_1\to x_2, x_2\to 0, x_3\to 0$). 
Let $\delta_{(n,m)}=n\delta_1$ and $\phi_{(n,m)}=m\phi_1$.
It is easy to see that both $\delta$ and $\phi$ are 
twisting systems of the $G$-graded Poisson algebra 
$(\Bbbk[x_1,x_2,x_3],0)$. We define $\delta+\phi$ by
$(\delta+\phi)_{(n,m)}=n\delta_1+m\phi_1$ for all $(n,m)\in 
{\mathbb Z}^2$. Since
$\delta_1$ and $\phi_1$ are not commuting, 
we see that \eqref{E2.1.1} fails for $\delta+\phi$.
\end{remark}

As noted before a derivation $\delta$ of $A$ is Poisson if 
and only if $d_{\pi}^1(\delta)=0$. By Definition 
\ref{xxdef0.4}, a graded derivation $\delta$ of a 
${\mathbb Z}$-graded Poisson algebra $A$ is {\it semi-Poisson} 
if $E\wedge d_{\pi}^1(\delta)=0$.

Next we show that the Poisson twisting systems induce an 
equivalence relation. Let $A$ be a $G$-graded commutative
algebra. Two graded Poisson structures $\pi$ and $\pi'$ 
on $A$ are called equivalent if $(A,\pi')$ is a graded twist of  
$(A,\pi)$. In this case we write $(A,\pi)\sim (A,\pi')$.

\begin{proposition}
\label{xxpro2.9} 
Suppose $G$ is cyclic. Then 
$\sim$ is an equivalence relation.
\end{proposition}

\begin{proof} It is clear that $(A,\pi)\sim (A,\pi)$
by taking the trivial Poisson twisting system $\delta$.
So $\sim$ is reflexive.

To prove the symmetry of $\sim$, we suppose that $(A,\pi')$ is a 
graded twist of $(A,\pi)$ by $\delta$. We claim that 
$-\delta$ is a Poisson twisting system of $(A,\pi')$. 
If this claim is proved, then it is obvious that
$(A,\pi')^{-\delta}=(A,\pi)$ as desired. It remains to show
that $-\delta$ satisfies \eqref{E2.1.1}, \eqref{E2.2.1}, 
and \eqref{E2.1.3}. The first two are easy. For the last one, 
we compute
$$\begin{aligned}
\delta_{|a|}(\langle b,c\rangle)&-\langle \delta_{|a|}(b),c\rangle
-\langle b, \delta_{|a|}(c)\rangle\\
&=\delta_{|a|}(\{b,c\}+b\delta_{|b|}(c)-c\delta_{|c|}(b))\\
&\quad -[\{\delta_{|a|}(b),c\}+\delta_{|a|}(b)
 \delta_{|b|}(c)-c\delta_{|c|}\delta_{|a|}(b)]\\
&\quad -[\{b,\delta_{|a|}(c)\}+b\delta_{|b|}
 \delta_{|a|}(c)-\delta_{|a|}(c)\delta_{|c|}(b)]\\
&= \delta_{|a|}(\{b,c\})-\{\delta_{|a|}(b),c\}-\{b,\delta_{|a|}(c)\}\\
&\quad
+\delta_{|a|}(b)\delta_{|b|}(c)+b\delta_{|a|}\delta_{|b|}(c)
-\delta_{|a|}(c)\delta_{|c|}(b)-c\delta_{|a|}\delta_{|c|}(b)\\
&\quad 
-\delta_{|a|}(b)\delta_{|b|}(c)+c\delta_{|c|}\delta_{|a|}(b)
-b\delta_{|b|}\delta_{|a|}(c)+\delta_{|a|}(c)\delta_{|c|}(b)\\
&
= \delta_{|a|}(\{b,c\})-\{\delta_{|a|}(b),c\}-\{b,\delta_{|a|}(c)\}
\end{aligned}
$$
which implies that

$$\begin{aligned}
p(\langle -,-\rangle,-\delta; a,b,c)
&=p(\{ -,-\},-\delta; a,b,c)=0.
\end{aligned}
$$
Therefore $-\delta$ is a Poisson twisting system of $A^{\delta}$
and $(A^{\delta})^{-\delta}=A$. So $\sim$ is symmetric.  

The proof above does not use the hypothesis that $G$ is cyclic.
The following part of the proof uses that hypothesis.

To prove the transitivity of $\sim$, we use the idea given in 
Example \ref{xxex2.6}. Suppose $\delta$ is a Poisson twisting 
system of $A$ and $\phi$ a Poisson twisting system of 
$A^{\delta}$. Let $\sigma:=\{\sigma_{g}:=\delta_{g}+\phi_{g}\mid g\in G\}$.

Since $G$ is cyclic, $\sigma_n=n\sigma_1$ by definition 
for all $n\in G$ (and there might be some restriction on ${\text{char}}\; 
\Bbbk$ when $G$ is finite). Therefore
\eqref{E2.1.1}-\eqref{E2.1.2} are obvious. Define 
$$\{a,b\}_{new}=\{a,b\}+a\sigma_{|a|}(b)-b\sigma_{|b|}(a)
=\langle a,b\rangle + a\phi_{|a|}(b)-b\phi_{|b|}(a).$$
Then $\{-,-\}_{new}$ is the Poisson bracket of $(A^{\delta})^{\phi}$.
By Theorem \ref{xxthm2.4}, $\sigma$ is a Poisson twisting 
system of $A$ and $A^{\sigma}=(A^{\delta})^{\phi}$. Therefore
$\sim$ is transitive.
\end{proof}

\begin{remark}
\label{xxrem2.10}
Let $G$ be ${\mathbb Z}^2$ and $A:=\Bbbk[x_0, x_1,x_2,x_3]$ 
with $\deg x_i=(1,0)$ for $i=0,1,2$ and $\deg x_3=(0,1)$.
We claim that $\sim$ is not an equivalence relation among the 
Poisson structures on $A$. We use $Y$ for the 
${\mathbb Z}^2$-graded Poisson algebra with trivial Poisson 
structure.

Let $\delta_1$ be the Poisson derivation of $Y$ sending 
$x_0\to 0$, $x_1\to x_2$, $x_2\to 0$, and $x_3\to 0$. Let 
$\phi_1$ be the Poisson derivation of $Y$ sending 
$x_0\to 0$, $x_1\to x_1$, $x_2\to 0$ and $x_3\to 0$.
We define two Poisson twisting systems as follows.
Let $\delta:=\{\delta_{(n,m)}=n \delta_1\}$
and $\phi:=\{\phi_{(n,m)}=m \phi_1\}$. Since 
$\delta_1$ and $\phi_1$ are Poisson derivations, 
it is easy to verify that $\delta$, $-\delta$ and $\phi$
are Poisson twisting systems. By Theorem \ref{xxthm2.4},
$X:=Y^{-\delta}$ is a Poisson algebra and by the first 
part of the proof of Proposition \ref{xxpro2.9},
$\delta$ is a Poisson twisting system of $X$ and 
$Y=X^{\delta}$. So we have $X\sim Y$. Let $Z=Y^{\phi}$.
Then $Y\sim Z$. We claim that $X\not\sim Z$. Suppose
on the contrary that $X\sim Z$. Then $Z=X^{h}$ for some 
Poisson twisting system $h$ of $X$. Applying 
\eqref{E2.3.1} to pairs of elements of the form 
$(x_i,x_j)$ for all $0\leq i,j\leq 3$, we see that
$h=\delta+\phi$. But it is clear that \eqref{E2.1.1}
fails for $\delta+\phi$ as $\delta_1\phi_1\neq \phi_1\delta_1$.
Therefore there is no Poisson twisting system $h$ such that
$Z=X^{h}$ as desired.

This example suggests that there should be a more general 
definition of twisting systems that induce an equivalence 
relation $\sim$.
\end{remark}

We conclude this section with some examples.

\begin{example}
\label{xxex2.11}
Here are three examples of twists of graded Poisson algebras.
\begin{enumerate}
\item[(1)]
Let $A$ be the Poisson polynomial ring $\Bbbk[x_1,\ldots,x_n]$ 
with trivial Poisson bracket. Consider $A$ as a 
${\mathbb Z}^n$-graded algebra with $\deg (x_i)=e_i$ 
where $e_i=(0,\ldots, 0,1,0,\ldots,0)\in {\mathbb Z}^n$ with 
$1$ in the $i$th position. Let $\{p_{ij}\mid 1\leq i< j 
\leq n\}$ be a subset of $\Bbbk$. For each $i$, define a 
${\mathbb Z}^n$-graded Poisson derivation $\delta_i$ by
$$\delta_i(x_j)=
\begin{cases} p_{ij} x_j & j>i,\\
0 & j\leq i.
\end{cases}
$$
For each $(a_1,\ldots,a_n)\in {\mathbb Z}^n$, let
$\delta_{(a_1,\ldots,a_n)}=\sum_{i=1}^n a_i \delta_i$. 
Since each $\delta_{(a_1,\ldots,a_n)}$ is a graded 
Poisson derivation of $A$, it is easy to see that  
$$\delta:=\{\delta_{(a_1,\ldots,a_n)}\mid 
(a_1,\ldots,a_n)\in {\mathbb Z}^n\}$$ 
is a twisting system
of $A$. By \eqref{E2.3.1}, the Poisson bracket of the new 
Poisson algebra $A^{\delta}$ is determined by
$$\langle x_i, x_j\rangle 
=x_i \delta_i(x_j)-x_j\delta_j(x_i)
=p_{ij}x_i x_j 
\quad
{\text{for all $i<j$.}}$$
\item[(2)]
Let $A$ be the Poisson polynomial ring $\Bbbk[x_1,\ldots,x_n]$ 
with trivial Poisson bracket. Consider $A$ as a 
${\mathbb Z}$-graded algebra with $\deg (x_i)=1$ for all $i$.
Let $\delta_1$ be a Poisson derivation of $A$ determined by
$$\delta_1(x_i)=\begin{cases} -x_{i-1} & i>1,\\
0& i=1.\end{cases}
$$
Let $\delta:=\{\delta_d:=d\delta_1\mid \forall\; d\in {\mathbb Z}\}.$
Since $\delta_1$ is a graded Poisson derivation, $\delta$ is a 
twisting system of $A$. By \eqref{E2.3.1}, the Poisson bracket 
of the new Poisson algebra $A^{\delta}$ is determined by
$$\langle x_i, x_j\rangle 
=x_i \delta(x_j)-x_j\delta(x_i)
=-x_i x_{j-1}+x_jx_{i-1}
\quad
{\text{for all $i<j$.}}$$
When $n=2$, the Poisson bracket of $A^{\delta}$ is determined
by
$$\langle x_2,x_1\rangle=x_1^2. $$
\item[(3)]
Suppose that ${\text{char}}\; \Bbbk=0$. Let $A$ be the $n$th 
Weyl Poisson algebra 
\[\Bbbk[x_1,\ldots,x_n,y_1,\ldots,y_n]\]
with Poisson bracket determined by
$$\{x_i, y_j\}=\delta_{ij}, \quad \{x_i,x_j\}=0, \quad
\{y_i,y_j\}=0$$
for all $1\leq i,j\leq n$. Here $\delta_{ij}$ is the usual 
Kronecker delta. Let $M=(m_{ij})_{n\times n}$ be any 
$n\times n$-matrix over $\Bbbk$. We consider $A$ as a
${\mathbb Z}$-graded Poisson algebra with $\deg(x_i)=1$ and
$\deg(y_i)=-1$ for all $1\leq i\leq n$. Let $\delta_1$ be 
the graded Poisson derivation of $A$ determined by
$$\delta_1(x_i)=\sum_{j=1}^n (-m_{ij}) x_j, 
\quad
\delta_1(y_i)=\sum_{j=1}^n m_{ji} y_j$$
for all $i$. Since $\delta_1$ is a graded Poisson derivation, 
$\delta:=\{n \delta_1\mid n \in {\mathbb Z}\}$ is a twisting 
system of $A$. By \eqref{E2.3.1}, the Poisson bracket of the 
new Poisson algebra $A^{\delta}$ is determined by 
$\langle x_i,x_j\rangle=0=\langle y_i,y_j\rangle$ and
\begin{equation}
\label{E2.11.1}\tag{E2.11.1}
\langle x_i, y_j\rangle 
=\delta_{ij}+ x_i \delta_1(y_j)-y_j\delta_{-1}(x_i)
=\delta_{ij}+x_i (\sum_{s=1}^{n} m_{sj} y_s)
-y_j(\sum_{s=1}^n m_{is} x_s)
\end{equation}
for all $i<j$.

We claim that $A^{\delta}$ is Poisson simple. To avoid some 
tedious analysis, we only give a sketch proof. The first 
step is to assume $\Bbbk$ is algebraically closed by replacing 
$\Bbbk$ by its algebraic closure. Since $\Bbbk$ is algebraically 
closed, we can assume $M$ is upper triangular, namely, 
$m_{ij}=0$ for all $i>j$. Let $M'$ be the matrix after deleting 
the diagonal entries of $M$ and we can similarly define a 
twisting system $\delta'$ using $M'$. By \eqref{E2.11.1}, one sees 
that $A^{\delta}=A^{\delta'}$. In other words, we can assume
that $M$ is strictly upper triangular. Now we can assume that 
$M$ is a block matrix with diagonal block equal to 
$$\begin{pmatrix} 0& 1& 0 &\cdots & 0\\
                0& 0& 1 &\cdots & 0\\
		\cdots &\cdots &\cdots& \cdots & \cdots\\
		            0& 0& 0 &\cdots & 1\\
                0& 0& 0 &\cdots & 0
\end{pmatrix}.
$$
To illustrate the idea, we assume $n=2$ and $M$ is given as above. 
(For general $M$, the proof is more complicated.) The Poisson 
structure of $A^{\delta}$ is determined by
$$\begin{aligned}
\langle x_1,x_2 \rangle &=0,\\
\langle y_1,y_2 \rangle &=0,\\
\langle x_1,y_1 \rangle &=1-x_2 y_1\\
\langle x_1,y_2 \rangle &=x_1y_1-x_2y_2\\
\langle x_2,y_1 \rangle &=0\\
\langle x_2,y_2 \rangle &=1+x_2 y_1
\end{aligned}
$$
Now we are ready to show that every nonzero Poisson 
ideal $I$ of $A^{\delta}$ contains $1$. Let $f$ be a 
nonzero element in $I$. We assume that the $y_2$-degree 
of $f$ is minimal and that $f\not\in \Bbbk$. Write
$f=\sum_{i\geq 0} f_i y_2^i$ where 
$f_i\in \Bbbk[x_1,x_2,y_1]$ for all $i$. If $f_i\neq 0$
for some $i>0$, then $H_{x_2}(f)=\sum_{i\geq 0} f_i
i y_2^{i-1}(1+x_2 y_1)\neq 0$. It is clear that the 
$y_2$-degree of $H_{x_2}(f)$ is smaller than the 
$y_2$-degree of $f$, yielding a contradiction. Therefore
$y_2$-degree of $f$ is zero, namely, $0\neq f\in I\cap
\Bbbk[x_1,x_2,y_1]$. Similarly, by using $H_{y_1}$, one can
show that there is an element $0\neq f\in \Bbbk[x_2,y_1]\cap I$.
Write $f=\sum_{i=0}^s g_i y_1^i$ where $g_i\in \Bbbk[x_2]$
for all $i$. Assume that $f$ has minimal $y_1$-degree.
If the $y_1$-degree of $f$, say $s$, is positive, then
$$H_{x_1}(f)+s f x_2 =\sum_{i=0}^s g_i i y_1^{i-1}
+\sum_{i=0}^s (s-i) g_i y_1^{i} x_2$$
which has $y_1$-degree $s-1$, yielding a contradiction.
Therefore $0\neq f\in \Bbbk[x_2]\cap I$. 
Write $f=\sum_{i=a}^b c_i x_2^i$ where $c_a c_b\neq 0$. 
Assume that $b-a\geq 0$ is minimal. Next we show that
$b-a=0$. Suppose to the contrary that $a<b$. Define
$\Phi(f)=-H_{y_2}(f)-\deg(f) f y_1$. Then 
$$\Phi(f)=\sum_{i=a}^b c_i i x_2^{i-1}+
\sum_{i=a}^b c_i (i-b) x_2^i y_1.$$
Then 
$$\begin{aligned}
H_{x_1}(\Phi(f))+\Phi(f) x_2
&=(\sum_{i=a}^b c_i (i-b) x_2^i)(1-x_2y_1)\\
&\quad +\sum_{i=a}^b c_i i x_2^{i}+
\sum_{i=a}^b c_i (i-b) x_2^i x_2 y_1\\
&=\sum_{i=a}^b c_i (i-b) x_2^i+\sum_{i=a}^b c_i i x_2^{i}\\
&=\sum_{i=a}^b c_i (2i-b) x_2^i\\
&=2 f'(x_2)x_2 -b f\in \Bbbk[x_2]\cap I.
\end{aligned}
$$
This implies that $x_2f'\in I$ and a possible linear combination 
of $f$ and $x_2f'$ gives a nonzero polynomial with smaller $b-a$, 
yielding a contradiction. Therefore we can assume $f=x^a$. 
In this case, $-H_{y_2}(f)=ax_2^{a-1}(1+x_2y_1)=ax_2^{a-1}+afy_1$. 
This implies that $x^{a-1}\in I$ when $a\neq 0$. An induction 
argument shows that $1\in I$ as required.
\end{enumerate}
\end{example}

\section{Proofs of Theorem \ref{xxthm0.2} and Corollary 
\ref{xxcor0.3}}
\label{xxsec3}

For the rest of the paper we assume that ${\rm{char}}\;
\Bbbk=0$. First we recall the definition of {\it divergence} 
of a skew symmetric $k$-derivation for $k\geq 0$
\cite[Sect. 4.4.3]{LPV}. 
A special case is given in Definition \ref{xxdef1.1}.
For $P\in {\mathfrak X}^p(A)$, the {\it internal 
product} $\iota_{P}$ is defined at the beginning
of Section \ref{xxsec1}.

Let $\nu$ be a volume form of $A$ which is in $\Omega^d(A)$
with $d$ being the top degree of nonzero $\Omega^d(A)$. We 
define the {\it star operator} 
$$\star(=\star_A): 
{\mathfrak X}^{\bullet}(A)\to \Omega^{d-\bullet}(A)$$
as follows: for each $k\ge 0$ and $Q\in {\mathfrak X}^k$, 
we set
$$\star_A Q:=\iota_{Q} \nu.$$
So $\star_A$ is a $\Bbbk$-linear map from 
${\mathfrak X}^{k}(A)\to \Omega^{d-k}(A)$
for each $k$. It follows from \eqref{E1.0.2} that
$\star_A$ is an $A$-linear map.

\begin{lemma}
\label{xxlem3.1} 
Let $B$ be a smooth affine domain of dimension $g$ with 
volume form $\nu$. Then $\star_B$ is an isomorphism.
\end{lemma}

\begin{proof}
To prove that $\star$ is an isomorphism, it suffices to show
that $\star_B \otimes_B B_{\mathfrak m}$ is an isomorphism 
for all maximal ideals ${\mathfrak m}$ of $B$. Let $A$ be 
the local ring $B_{\mathfrak m}$.  Then $A$ is a regular local 
ring of global dimension, Krull dimension and transcendence 
degree $n$. Since all the operations commute with the localization,  
$$\star_B \otimes_B B_{\mathfrak m}=\star_B \otimes_B A
=\star_A.$$

Now we assume that $A$ is local with maximal ideal ${\mathfrak m}$
generated by $\{x_1,\ldots,x_n\}$. Then $\Omega^1(A)$ is a free 
module of rank $g$. Write $\Omega^1(A)=\oplus_{i=1}^n A dx_i$. 
Then for each $k\geq 0$, $\Omega^k(A)$ is a free $A$-module with 
basis
$$\{d_{i_1,\cdots,i_k}:
=d x_{i_1} \wedge \cdots \wedge d x_{i_k}\mid 1\leq i_1<
\cdots <i_k\le n\}$$ and that, via \eqref{E1.0.1}, 
${\mathfrak X}^k(A)$ is a free $A$-module with basis as 
in \eqref{E1.0.1}
$$\{\partial_{i_1,\ldots,i_k}:
={\frac{\partial \;}{\partial x_{i_1}}} \wedge \cdots 
\wedge {\frac{\partial \;}{\partial x_{i_k}}}\mid 1\leq i_1<
\cdots <i_k\leq n\}.$$ 
Recall that $\nu=a\,d_{1,2,\ldots,n}$ for an invertible
element $a\in A$. (Usually we write $\nu_0=d_{1,2,\ldots,n}$). 
By definition,
$$\begin{aligned}
\star_A \partial_{i_1,\ldots,i_k} &=
\iota_{\partial_{i_1,\ldots,i_k}} (a\nu_0)
 = a\iota_{d_{i_1,\ldots,i_k}} 
 (d_{1,2,\cdots,n})\\
&=\sum_{\sigma\in {\mathbb S}_{k,n-k}}
 sgn(\sigma) a \partial_{i_1,\ldots,i_k}
 [x_{\sigma(1)},\ldots,x_{\sigma(k)}]
 d_{1,\ldots, \widehat{\sigma(1)},\ldots, 
 \widehat{\sigma(k)},\ldots,n}\\
&=\pm a\, d_{1,\ldots, \widehat{i_1},\cdots, \widehat{i_k},\ldots,n}
\end{aligned}
$$
where $\pm 1=sgn(\{i_1,\ldots,i_k,1,\ldots, 
\widehat{i_1},\ldots, \widehat{i_k},\ldots,n\})$.
Therefore $\star_A$ is an isomorphism as desired.
\end{proof}

\begin{definition}
\label{xxdef3.2}
We say $A$ is a {\it standard} Poisson algebra if $A$ is an 
affine smooth ${\mathbb Z}$-graded Poisson domain with a 
homogeneous volume form $\nu$ (with $\deg (\nu)$ not 
necessarily zero) and $\star$ is an isomorphism. 
\end{definition}

Note that every polynomial Poisson algebra $\Bbbk[x_1,\ldots,x_n]$
is standard (even when ${\rm{char}}\; \Bbbk>0$). Lemma 
\ref{xxlem3.1} provides another class of such 
algebras. Now we assume that $A$ is standard of dimension $n$. By 
\cite[Sect. 4.4.3]{LPV}, the {\it divergence} operator with 
respect to the volume form $\nu$ is a graded $\Bbbk$-linear 
map of degree $-1$,
$$\divv: 
{\mathfrak X}^{\bullet}(A)\to {\mathfrak X}^{\bullet-1}(A),$$
which makes the following diagram commutes
$$\begin{CD}
{\mathfrak X}^{\bullet}(A)@> \star >> \Omega^{n-\bullet}(A)\\
@V \divv VV @VV d V\\
{\mathfrak X}^{\bullet-1}(A)@> \star >> \Omega^{n-\bullet+1}(A)
\end{CD}
$$

Since $A$ is standard, the star operator $\star$ is an $A$-linear
isomorphism. Now we have the following lemmas which were proved
in \cite[Sect. 4.4.3]{LPV}.

\begin{lemma} \cite[Proposition 4.16]{LPV}
\label{xxlem3.3}
Suppose $(A,\pi)$ is standard with volume form $\nu$. Let 
$\delta$ and $\phi$ be two derivations of $A$. Then
$$\divv (\delta \wedge \phi)
=\divv(\phi) \delta-\divv(\delta) \phi -[\delta,\phi].$$
\end{lemma}

The following lemma gives another proof of Lemma \ref{xxlem0.1}.

\begin{lemma} \cite[Proposition 4.17]{LPV}
\label{xxlem3.4}
Suppose $(A,\pi)$ is standard with volume form $\nu$.
Let ${\mathbf m}$ be the modular derivation of $A$. Then 
$${\mathbf m}=-\divv (\pi).$$
As a consequence, the divergence of ${\mathbf m}$ is zero.
\end{lemma}

\begin{question}
\label{xxque3.5}
It is not clear how to handle nonaffine smooth domain $A$ 
as the proof of Lemma \ref{xxlem3.1} uses the fact $A$ is affine.
\end{question}

\begin{theorem}
\label{xxthm3.6}
Suppose $(A,\pi)$ is standard with volume form $\nu$.
Let $\delta$ be a graded semi-Poisson derivation of $A$.
Let ${\mathbf m}$ (respectively, ${\mathbf n}$) be the 
modular derivation of $A$ (respectively $A^{\delta}$). 
Then 
$${\mathbf n}={\mathbf m}+(\divv E) \delta 
-(\divv \delta) E.$$
\end{theorem}

\begin{proof} Let $\pi'$ be the Poisson structure of 
$A^{\delta}$. By \eqref{E2.3.1},
$$\pi'=\pi+E\wedge \delta.$$
By Lemmas \ref{xxlem3.3} and \ref{xxlem3.4}, we have 
$$\begin{aligned}
{\mathbf n}&=-\divv (\pi')=
-\divv(\pi)-\divv (E\wedge \delta)\\
&={\mathbf m} +\divv(E)\delta -\divv(\delta)E -[\delta,E].
\end{aligned}
$$
The assertion follows as $[\delta,E]=0$ for each graded 
derivation $\delta$.
\end{proof}

\begin{proof}[Proof of Theorem \ref{xxthm0.2}]
Let $A$ be a ${\mathbb Z}$-graded Poisson algebra 
$\Bbbk[x_1,\ldots,x_n]$ with each $x_i$ homogeneous.  By 
Lemma \ref{xxlem1.2}(5), $\divv(E)=\deg (\nu)=
\sum_{i=1}^n \deg(x_i)=:{\mathfrak l}$. Now the 
assertions follow from Theorem \ref{xxthm3.6}.
\end{proof}

\begin{remark}
\label{xxrem3.7}
There is a different proof of Theorem \ref{xxthm0.2} 
without using Theorem \ref{xxthm3.6} (details are 
omitted). In fact, the different proof does not
use the hypothesis that ${\text{char}}\; \Bbbk=0$.
\end{remark}

The following result is an immediate consequence 
of Theorem \ref{xxthm3.6}.

\begin{theorem}
\label{xxthm3.8}
Suppose $(A,\pi)$ is standard with volume $d$-form $\nu$. 
Assume that $\divv(E)\in \Bbbk$ is nonzero. Let 
${\mathbf m}$ be the modular derivation of $A$ and let 
$\delta=-{\frac{1}{\divv(E)}} {\mathbf m}$. Then 
$(A^{\delta}, \pi')$ is unimodular and 
$$\pi=\pi'+\frac{1}{\divv(E)}\, E\wedge {\mathbf m.}$$
In particular, we have $\mathcal L_{\delta}(\alpha)=0$ 
where $\alpha=\star \pi'$ is the closed differential 
$(d-2)$-form associated with the unimodular Poisson 
structure $\pi'$ on $A$ .
\end{theorem}

\begin{proof} 
Let ${\mathbf n}$ be the modular derivation of 
$(A^{\delta},\pi')$. By Theorem \ref{xxthm3.6} and the fact
$\delta=-\frac{1}{\divv(E)} {\mathbf m}$, 
$${\mathbf n}={\mathbf m}+\divv(E) \delta -(\divv \delta) E
=-(\divv \delta) E=0$$
where the last equation follows from $\divv {\mathbf m}=0$
[Lemma \ref{xxlem0.1}]. Therefore $(A^{\delta}, \pi')$ is 
unimodular. By \eqref{E2.3.1}, for all $a,b\in A$,
$$\begin{aligned}
\pi'(a,b)&=\langle a,b\rangle
=\{a,b\}+a\delta_{|a|}(b)-b\delta_{|b|}(a)\\
&=\pi(a,b)+|a|a \delta(b)-|b|b\delta(a)
=\pi(a,b)+E(a)\delta(b)-\delta(a)E(b)
\end{aligned}
$$
which implies that
$$\pi'=\pi+E\wedge \delta$$
which is equivalent to the above assertion. Finally 
by \cite[Proposition 3.11(2)]{LPV} we have
\begin{align*}
\mathcal L_\delta(\alpha)=
&\mathcal L_\delta(\iota_{\pi'}(\nu))
=\iota_{\pi'}(\mathcal L_\delta(\nu))+\iota_{[\delta,\pi']_S}(\nu)\\
=&\iota_{\pi'}(\divv (\delta)\nu)+\iota_{[\delta,\pi']_S}(\nu)
=\iota_{[\delta,\pi']_S}(\nu)
\end{align*}
One can easily check that $\delta$ is also a Poisson 
derivation of $(A^\delta,\pi')$. So 
$[\delta,\pi']_S=-d_{\pi'}(\delta)=0$ and we get 
$\mathcal L_\delta(\alpha)=0$.
\end{proof}

\begin{proof}[Proof of Corollary \ref{xxcor0.3}]
Let $A$ be a ${\mathbb Z}$-graded Poisson algebra 
$\Bbbk[x_1,\cdots,x_n]$ with each $x_i$ homogeneous.  
By lemma \ref{xxlem1.2}(5),
${\mathfrak l}=\sum_{i=1}^n \deg(x_i)$. Now the 
assertions follow from Theorem \ref{xxthm3.8}.
\end{proof}

\section{Rigidity of graded twisting}
\label{xxsec4}
Let $A$ be a ${\mathbb Z}$-graded Poisson algebra. Recall
that the set of graded semi-Poisson derivations (resp.
graded Poisson derivations) of $A$ of degree 0 is denoted
by $Gspd(A)$ (resp. $Gpd(A)$). We first prove Theorem 
\ref{xxthm0.5}.

\begin{lemma}
\label{xxlem4.1}
Let $A$ be a ${\mathbb Z}$-graded Poisson algebra
$\Bbbk[x_1,\ldots,x_n]$ with $\deg (x_i)>0$ for every $i$. 
Suppose that $A$ is unimodular and that  
${\mathfrak l}:=\sum_{i=1}^n \deg(x_i)$ is a nonzero 
element in $\Bbbk$. If $\delta$ is a semi-Poisson 
derivation of $A$, then $\delta$ is a Poisson derivation
of $A$. Namely, $Gspd(A)=Gpd(A)$. 
\end{lemma}

\begin{proof}
Since $\deg (x_i)>0$ for all $i$, by Lemma \ref{xxlem1.2}(1,3), 
both $\divv(\delta)$ and ${\mathfrak l}=\divv(E)$ are in $\Bbbk$.
Let $B$ be the twist $A^{\delta}$ with modular derivation
${\mathbf n}$.  By Theorem \ref{xxthm0.2},
$${\mathbf n}={\mathbf m}+{\mathfrak l} \delta-\divv(\delta) E
={\mathfrak l} \delta-\divv(\delta) E.$$
Since ${\mathbf n}$ and $E$ (and $\divv(\delta) E$) are Poisson 
derivations of $B$, we have that ${\mathfrak l} \delta$ (and hence 
$\delta$) is a Poisson derivation of 
$B$. Let $\langle-,-\rangle'$ (respectively,
$\{-,-\}$) be the Poisson structure of $B$ (respectively, $A$).
By \eqref{E2.3.1}, we have
$$\{-,-\}=\langle -,-\rangle-E\wedge \delta.$$
Then, for all homogeneous elements $a,b\in A$, 
$$\begin{aligned}
\delta(\{a,b\})-&\{\delta(a),b\}-\{a,\delta(b)\}\\
&=\delta(\langle a,b\rangle)-
  \langle\delta(a),b\rangle-\langle a,\delta(b)\rangle\\
&\quad
 -\delta((E\wedge \delta)[a,b])+E\wedge 
 \delta[\delta(a),b]+E\wedge \delta[a,\delta(b)]\\
&=0-\delta(|a|a\delta(b)-|b|b\delta(a))\\
&\quad +(|a|\delta(a)\delta(b)-|b|b\delta^2(a))
+(|a|a\delta^2(b)-|b|\delta(b)\delta(a))\\
&=-|a|\delta(a)\delta(b)-|a|a\delta^2(b)
 +|b|\delta(b)\delta(a)+|b|b\delta^2(a)\\
&\quad +(|a|\delta(a)\delta(b)-|b|b\delta^2(a))
+(|a|a\delta^2(b)-|b|\delta(b)\delta(a))\\
&=0.
\end{aligned}
$$
Therefore $\delta$ is a Poisson derivation of $A$.
\end{proof}

\begin{lemma}
\label{xxlem4.2}
Let $B$ be a twist of $A$. Then $Gspd(A)=Gspd(B)$.
\end{lemma}

\begin{proof}
Write $B=A^{\delta}$. So $B=A$ as a commutative algebra.
Let $\pi$ (respectively, $\pi'$) be the Poisson bracket of
$A$ (respectively, $B$). 

For every derivation $\phi$ of $A$, $[E,\phi]_{S}
=[E,\phi]=0$. For any two derivations $\phi_1,\phi_2$,
we have
$$\begin{aligned}
\; [E\wedge \phi_1,E\wedge \phi_2]_S
&=\pm [E, E\wedge \phi_2]_S\wedge \phi_1
\pm E\wedge [\phi_1, E\wedge \phi_2]_S\\
&=\pm ([E,E]_S\wedge \phi_2\pm [E,\phi_2]_S\wedge E)\wedge \phi_1\\
&\qquad \pm E\wedge ([\phi_1, E]_S\wedge \phi_2\pm [\phi_1,\phi_2]_S\wedge E)\\
&=0.
\end{aligned}
$$

Let $\phi$ be a semi-Poisson derivation of $A$. By definition,
$$[E\wedge \phi, \pi]_{S}=0.$$
Then 
$$\begin{aligned}
\; [E\wedge \phi, \pi']_{S}
&=[E\wedge \phi, \pi+E\wedge \delta]_S\\
&=[E\wedge \phi, \pi]_S+[E\wedge \phi, E\wedge \delta]_S\\
&=0.
\end{aligned}
$$ 
Therefore $\phi$ is a semi-Poisson derivation of $B$. 
\end{proof}

\begin{proof}[Proof of Theorem \ref{xxthm0.5}]
Part (1) is Lemma \ref{xxlem4.1} and part (2) is Lemma 
\ref{xxlem4.2}.

(3) By Corollary \ref{xxcor0.3} and part (2), we may assume
that $A$ is unimodular. By part (1), $Gspd(A)$ is the 
$\Bbbk$-vector space of graded Poisson derivations $A$. It is
well-known that it is a Lie algebra. Let $\phi$ be any 
Poisson derivation of $A$. It is clear that $\phi$ is determined by
$\{\phi(x_i)\}_{i=1}^n$. Therefore $Gspd(A)$ is finite dimensional.
\end{proof}

One of the main definitions in this paper is the following.

\begin{definition}
\label{xxdef4.3}
Let $A$ be a ${\mathbb Z}$-graded Poisson algebra. 
\begin{enumerate}
\item[(1)]
The {\it rigidity of graded twisting} (or simply {\it rigidity}) 
of $A$ is defined to be
$$rgt(A)=1-\dim_{\Bbbk} Gspd(A).$$
\item[(2)]
We say $A$ is {\it rigid} if $rgt(A)=0$.
\item[(3)]
We say $A$ is {\it $(-1)$-rigid} if $rgt(A)=-1$.
\end{enumerate}
\end{definition}

Note that this notion of rigidity is different from
the rigidity defined in \cite[Definition 0.1]{GVW} 
and other papers.

It follows from Lemma \ref{xxlem4.2} that
\begin{equation}
\label{E4.3.1}\tag{E4.3.1}
rgt(A)=rgt(A^{\delta}).
\end{equation}

Other basic facts about $rgt(A)$ are listed in the following lemma.

\begin{lemma}
\label{xxlem4.4}
Let $A$ be a ${\mathbb Z}$-graded Poisson algebra with $A_i\neq 0$
for some $i\neq 0$. In parts {\rm{(2)-(6)}}, we further assume that 
$A$ is $\Bbbk[x_1,\ldots,x_n]$ with $\deg x_i>0$ for all $i$. 
\begin{enumerate}
\item[(1)]
Suppose that $rgt(A)=0$. Then every graded twist of $A$ is isomorphic 
to $A$.
\item[(2)]
If $A$ is not unimodular, then $rgt(A)\leq -1$.
\item[(3)]
If $rgt(A)=0$, then $A$ is unimodular.
\item[(4)]
If $rgt(A)\neq 0$, then $\dim_{\Bbbk} Gspd(A) \geq \dim_{\Bbbk} Gpd(A)\geq 2$.
\item[(5)]
If $rgt(A)=-1$, then $\dim_{\Bbbk} Gspd(A)= \dim_{\Bbbk} Gpd(A)=2$.
\item[(6)]
Let $\{A(a)\}_{a\in \Bbbk}$ be a family of Poisson polynomial algebras such that
{\rm{(i)}} $A(a)$ is a Poisson twist of $A(a')$ for all $a,a'\in \Bbbk$ and that
{\rm{(ii)}} there is an $a_0$ such that $A(a_0)$ is unimodular. If $\dim_{\Bbbk}
Gpd(A(a_0))=2$, then $rgt(A(a))=-1$ for all $a$.
\item[(7)]
If $rgt(A)=0$ and $A$ is a connected graded domain, then every 
Poisson normal element of $A$ is Poisson central.
\item[(8)] If $\dim_\Bbbk Gpd(A)=1$ and $\divv(E)\neq 0$, then 
$rgt(A)=0$. 
\end{enumerate}
\end{lemma}

\begin{proof}
(1) Since $A_i\neq 0$ for some $i$, the Euler derivation $E$ is not zero.
Since $rgt(A)=0$, $Gspd(A)=\Bbbk E$. Let $B$ be a graded twist of $A$. Then 
$B=A^{\delta}$ where $\delta\in Gspd(A)$. Let $\langle -,-\rangle$ be the 
Poisson bracket of $B$. Since $\delta=\alpha E$ for some $\alpha\in \Bbbk$, 
one sees from \eqref{E2.3.1} that $\langle a,b\rangle=\{a,b\}$ where
$\{a,b\}$ is the original Poisson bracket of $A$. The assertion follows.

(2) Since $A$ is not unimodular, the modular derivation ${\mathbf m}$ 
is not in $\Bbbk E$, as $\divv (E)=\sum_{i=1}^n \deg x_i\neq 0$ and 
$\divv({\mathbf m})=0$ [Lemma \ref{xxlem0.1}]. Therefore 
$\dim_\Bbbk Gspd(A)\geq 2$. The assertion follows.

(3) This is equivalent to (2).

(4) By definition, it is clear that $ \dim_{\Bbbk} Gspd(A) \geq 
\dim_{\Bbbk} Gpd(A)$. It remains to show $\dim_{\Bbbk} Gpd(A)
\geq 2$. If $A$ is unimodular, then, by Lemma \ref{xxlem4.1},
$$\dim_{\Bbbk} Gpd(A)=\dim_{\Bbbk} Gspd(A)=1-rgt(A)\geq 2.$$
Now we assume that $A$ is not unimodular with nonzero 
modular derivation ${\mathbf m}$. Since $\divv(E)\neq 0$ and
$\divv({\mathbf m})=0$, the $\Bbbk$-dimension of $Gpd(A)$ is
at least 2 as desired.

(5) By definition, $\dim_{\Bbbk} Gspd(A)=1-rgt(A)=2$. The assertion
follows from part (4).

(6) It follows from Lemmas \ref{xxlem4.1} and \ref{xxlem4.2} that we have 
$$rgt(A(a))=rgt(A(a_0))=1-\dim_{\Bbbk} Gspd(A(a_0))
=1-\dim_{\Bbbk} Gpd(A(a_0))=-1.$$
The assertion follows from (4).

(7) We only need to consider a homogeneous Poisson normal
element $f$ of positive degree. Note the log-Hamiltonian 
derivation $LH_f:=f^{-1}\{f,-\}$ is a Poisson derivation of 
degree $0$. Suppose $f$ is not central. Then $LH_f(f)\neq 0$. 
Note that $LH_f$ is clearly not the Euler derivation. Therefore 
$rgt(A)\leq -1$, yielding a contradiction. The assertion follows.

(8) We know $Gpd(A)$ is spanned by the Euler derivation $E$, 
which is not the modular derivation by Lemma \ref{xxlem0.1}. 
This implies that $A$ is unimodular and the result follows by 
Lemma \ref{xxlem4.2}.
\end{proof}

Examples of $rgt(A)$ will be given in the next 2 sections.

\section{Examples and comments}
\label{xxsec5}

Note that the graded twists in the associative algebra setting 
has an important property, namely, a graded algebra $R$ and 
its twist have isomorphic corresponding graded module categories 
\cite{Zh}. So we are wondering if a similar result holds in the 
Poisson setting. The following example shows that this is not the 
case. 

\begin{example}
\label{xxex5.1}
Let $A$ be the Poisson algebra $\Bbbk[x_1,x_2]$ with trivial
Poisson structure. Let $\delta$ be the Poisson derivation of
$A$ determined by
$$\delta(x_1)=0, \quad {\text{and}} \quad \delta(x_2)=x_2.$$
Let $B$ be the graded twist of $A$ by $\delta$, namely, 
$B=A^{\delta}$. By definition, $B$ is a Poisson algebra
$\Bbbk[x_1,x_2]$ with Poisson bracket determined by
$$\{x_1,x_2\}=x_1x_2.$$

Let $U(A)$ denote the Poisson enveloping algebra of $A$
\cite{Ba3}. Since $A$ has the trivial Poisson structure, $U(A)$ 
is the commutative polynomial ring $\Bbbk[x_1,x_2,y_1,y_2]$.
Let $U(B)$ be the Poisson enveloping algebra of $B$. We claim 
that $U(B)$ is not a graded twist of $U(A)$ in the sense of
\cite{Zh}. 

Suppose on the contrary that $U(B)$ is a graded twist of $U(A)$ 
in the sense of \cite{Zh}. Then, by \cite[Theorem 1.1]{Zh}, 
\begin{equation}
\label{E5.1.1}\tag{E5.1.1}
{\text{GrMod-}}U(A)\cong {\text{GrMod-}}U(B).
\end{equation}
Let $D(A)$ (respectively, $D(B)$) be the degree zero part of 
graded quotient ring of $U(A)$ (respectively, $U(B)$). Then 
it follows from \eqref{E5.1.1} that $D(A)\cong D(B)$. Since 
$U(A)$ is commutative, $D(A)$ is commutative. Thus $D(B)$ is 
commutative. Next we prove that $D(B)$ is not commutative, so 
we obtain a contradiction. By \cite[Theorem 2.2]{Ba3}, $U(B)$ 
is generated by four elements $x_1,x_2, \delta_1, \delta_2$ 
and subject to 6 relations
$$\begin{aligned}
x_1 x_2&=x_2x_1,\\
\delta_1 x_1&=x_1 \delta_1,\\
\delta_1 x_2&=x_2 \delta_1+x_1x_2,\\
\delta_2 x_1&=x_1 \delta_2-x_1x_2,\\
\delta_2 x_2&=x_2 \delta_2,\\
\delta_2 \delta_1&=\delta_1 \delta_2+x_2\delta_1+x_1\delta_2.
\end{aligned}
$$
Let $a=x_1x_2^{-1}$ and $b=\delta_2 x_2^{-1}$ which are 
elements in $D(B)$. It follows from the six relations that 
$$ba=ab-a.$$ 
So $D(B)$ is not commutative, yielding a contradiction.
Therefore $U(B)$ is not a graded twist of $U(A)$.

As a consequence, the category of graded Poisson modules over 
$A^{\delta}$, denoted by ${\text{GrPMod-}}A^{\delta}$ is not 
equivalent to the category of graded Poisson modules over 
$A$, denoted by ${\text{GrPMod-}}A$. That is,
\[
{\text{GrPMod-}}A^{\delta}\not\cong {\text{GrPMod-}}A.
\]
\end{example}

\begin{remark}
\label{xxrem5.2}
When $A$ is a connected graded Poisson algebra with 
$A_{i}\neq 0$ for some $i>0$. Then $PH^1(A)$ is also graded. 
Since $(PH^1(A))_{0}\cong Gpd(A)$, we have 
$rgt(A)\leq 1-\dim_{\Bbbk} (PH^1(A))_{0}$. If 
$\dim_{\Bbbk}(PH^1(A))_{0}=1$, then $rgt(A)=0$ by Lemma \ref{xxlem4.4}(8).  
Therefore we can obtain information about $rgt(A)$ if we know
$PH^1(A)$.
\end{remark}

\begin{remark}
\label{xxrem5.3}
Let $A$ be the quadratic Poisson algebra corresponding to the
$4$-dimensional Sklyanin algebra. The Poisson (co)homologies 
of $A$ have been computed in \cite{Pe1}. By the Hilbert
series of Poisson homologies of $A$ given in \cite[p.1154]{Pe1}
and the Poincar{\' e} duality, both $PH^0(A)$ and $PH^1(A)$ have 
Hilbert series $\frac{1}{(1-t^2)^2}$. By Remark \ref{xxrem5.2},
$rgt(A)=0$. Therefore $A$ is rigid of graded twisting. By the 
correspondence between quadratic Poisson algebras and the 
homogeneous coordinate rings of quantum spaces, the corresponding  
Sklyanin algebra is considered as rigid in some sense. Further,
since $h_{PH^1(A)}(t)=h_{PH^0(A)}(t)=h_Z(t)=\frac{1}{(1-t^2)^2}$,
$A$ is $PH^1$-minimal in the sense of Definition \ref{xxdef7.3}(1).

Note that the Poisson (co)homologies of the quadratic Poisson 
algebra corresponding to the $3$-dimensional Sklyanin algebra were 
computed in \cite{Pe2, Pi1, VdB}, an argument similar to the above shows 
that $rgt(A)=0$ (namely, $A$ is rigid of graded twisting) and $A$ is 
$PH^1$-minimal. We will give an elementary and direct computation 
of this $rgt(A)$ in Example \ref{xxex6.6}(Case 3). 
\end{remark}

\begin{remark}
\label{xxrem5.4}
Let $n\geq 2$ and $a\in \mathbb C$. Let 
$A(n,a)=\mathbb C[x_1,\ldots,x_n]$ be the family of 
Poisson polynomial algebras studied in \cite{LS}. Then 
one can check for each fixed $n\ge 2$ the family 
$A(n,a)$ satisfies all the assumptions stated in 
Lemma \ref{xxlem4.4}(6) with $a_0=\frac{(n+2)(1-n)}{2(n+1)}$ 
such that $rgt(A(n,a))=-1$. Computations are omitted. In 
particular for any $a,a'\in \mathbb C$, $A(n,a)$ is a 
Poisson twist of $A(n,a')$, which is a Poisson version of 
\cite[Theorem 4.2]{LS}. 
\end{remark}

We will compute $rgt$ for some classes of Poisson algebras. 
Here is a warmup.  

\begin{example}
\label{xxex5.5}
Let $\deg x_i=i$ for $i=1,2,3$. Let $\Omega= x^6+y^3+z^2+
\lambda xyz$ where $\lambda\in \Bbbk$. For any $\lambda$, 
$\Omega$ is irreducible. Define a Poisson structure on 
$A:=\Bbbk[x,y,z]$ by
$$\{f,g\}
:= \det \begin{pmatrix} 
\Omega_x & f_x & g_x\\
\Omega_y & f_y & g_y\\
\Omega_z & f_z & g_z
\end{pmatrix}$$
for all $f,g\in A$. It is easy to see that
\begin{align}
\label{E5.5.1}\tag{E5.5.1}
\{x,y\}&= \Omega_z=2z+\lambda xy,\\
\label{E5.5.2}\tag{E5.5.2}
\{x,z\}&= -\Omega_y=-(3y^2+\lambda xz),\\
\label{E5.5.3}\tag{E5.5.3}
\{y,z\}&= \Omega_x=6x^5+\lambda yz
\end{align}
and that $A$ is unimodular. If $\lambda^6\neq 6^3$, then 
$A_{sing}:=A/(\Omega_x,\Omega_y,\Omega_z)$ is finite dimensional.
In this case, $\Omega$ has isolated singularity.  
Let $\delta$ be a graded Poisson derivation. Then 
$$\begin{aligned}
\delta(x)&= c_1 x,\\
\delta(y)&= c_2 y+ c_3 x^2,\\
\delta(z)&= c_4 z+ c_5 x^3+ c_6 xy.
\end{aligned}
$$
Subtracting by $c_1 E$, we may assume that $c_1=0$. Applying $\delta$ 
(with $c_1=0$) to \eqref{E5.5.1}, we 
obtain that
$$
c_2 (2z+\lambda xy) 
=2(c_4 z+ c_5 x^3+ c_6 xy)+\lambda x (c_2 y+ c_3 x^2),
$$
which implies that $c_2=c_4$, $c_6=0$, and $2c_5+\lambda c_3=0$.
Applying $\delta$ 
(with $c_1=0$ and $c_6=0$) to \eqref{E5.5.2}, we 
obtain that
$$
-c_4(3y^2+\lambda xz)
=-6y(c_2 y+ c_3 x^2)-\lambda x (c_4 z+ c_5 x^3),
$$
which implies that $c_2=c_3=c_4=c_5=0$. Therefore $\delta=0$.
This means that $rgt(A)=0$. By Lemma \ref{xxlem4.4}(1), $A$ does 
not have any non-trivial twists.

In general, when $\Omega$ has isolated singularity, the fact 
that $rgt(A)=0$ also follows from the Poisson cohomology 
computation given in \cite[Proposition 4.5]{Pi1} (after matching 
up the notations). The same idea applies to the algebra in 
Example \ref{xxex6.6}(Case 3).
\end{example}

\section{Some computations of $rgt$}
\label{xxsec6}

In this section we compute $rgt$ for all quadratic Poisson 
structures on $\Bbbk[x,y,z]$. Some of the computations have 
been done by other researchers in different language (for 
example, some are hidden inside in Poisson cohomology 
computation), but we provide all details of computations of 
$rgt$ for completeness. The classification of all quadratic 
Poisson structures on $\Bbbk[x,y,z]$ were given in 
\cite{DH, DML, LX}. 

First we fix some notations. Let $\Bbbk$ be an algebraically 
closed field of characteristic zero (one might assume 
$\Bbbk={\mathbb C}$ if necessary). Let 
$V=A_1=\Bbbk x+\Bbbk y+\Bbbk z$ and let $\{-,-\}$ be a 
quadratic Poisson bracket of $A:=\Bbbk[x,y,z]=\Bbbk [V]$. Let 
$f$ be a graded Poisson derivation of $(A,\{-,-\})$. Let 
$W=\{V,V\}$. It is clear that
\begin{equation}
\label{E6.0.1}\tag{E6.0.1}
f(W)=
f(\{V,V\})\subseteq \{f(V),V\}+\{V,f(V)\}\subseteq \{V,V\}=W.
\end{equation}
Write
\begin{equation}
\label{E6.0.2}\tag{E6.0.2} 
f(x)=a_1 x+ a_2 y+ a_3 z,
f(y)=b_1 x+b_2 y+b_3 z, f(z)=c_1 x+c_2 y+c_3 z.
\end{equation}
After replacing $f$ by $f-a_1 E$, we can further assume that
\begin{equation}
\label{E6.0.3}\tag{E6.0.3} 
a_1\ {\text{in \eqref{E6.0.2} is zero.}} 
\end{equation}

If $\{-,-\}$ is unimodular, then it comes from a ``potential''
$\Omega\in A_3$. One can classify $\Omega$ as follows: (a):
$\Omega$ is a product of three linear terms, (b): $\Omega$ is a 
product of a linear term and an irreducible polynomial of 
degree $2$, and (c): $\Omega$ is irreducible of degree $3$. 
This classification is classical and well-known, see for 
example \cite{BM, KM, Ri}.  

The following four examples deal with the first case, namely, 
$\Omega$ is a product of three linear terms.

Define $A_{sing}$ to be $A/(\Omega_x, \Omega_y,\Omega_z)$.

\begin{example}
\label{xxex6.1}
Let $\Omega=x^3$. Then
$$\begin{aligned}
\{x,y\}&=\Omega_z=0,\\
\{z,x\}&=\Omega_y=0,\\
\{y,z\}&=\Omega_x=3x^2.
\end{aligned}
$$
It is clear that $\Kdim(A_{sing})=2$. Let $f$ be a Poisson 
derivation of $A$. By \eqref{E6.0.1}
$$2x f(x)=f(x^2) \in f(W) \subseteq W=\Bbbk x^2.$$
Then $f(x)=ax$ for some $a\in \Bbbk$. By \eqref{E6.0.3},
we may assume that $f(x)=0$. Retain the notations in 
\eqref{E6.0.2}. Applying $f$ to $\{y,z\}=3x^2$ 
implies that $b_2+c_3=0$ with $b_1, b_3, c_1, c_2$ free. 
Therefore $rgt(A)=-5$. One can check that every Poisson 
normal element of $A$ is Poisson central. 
\end{example}

\begin{example}
\label{xxex6.2}
Let $\Omega=x^2y$. Then
\begin{align}
\notag \{x,y\}&=\Omega_z=0,\\
\label{E6.2.1}\tag{E6.2.1}
\{z,x\}&=\Omega_y=x^2,\\
\notag \{y,z\}&=\Omega_x=2xy.
\end{align}
It is clear that $\Kdim(A_{sing})=2$. Let $f$ be a Poisson 
derivation of $A$. By \eqref{E6.0.1}, we have
$$\Bbbk x^2+\Bbbk xy=W \supseteq f(W)=
\Bbbk (2x f(x))+\Bbbk (f(x)y+xf(y)).$$
Then $f(x)y$ does not have terms $y^2$ and $yz$. So
$f(x)=ax$ for some $a\in \Bbbk$, and by \eqref{E6.0.3}, 
we may assume that $f(x)=0$. Using the notations in 
\eqref{E6.0.2}, then \eqref{E6.2.1} implies that 
$b_1=b_3=c_3=0$ with $b_2, c_1, c_2$ free. Therefore
$rgt(A)=-3$.
\end{example}

\begin{example}
\label{xxex6.3}
Let $\Omega=xyz$. Then
\begin{align}
\notag \{x,y\}&=\Omega_z=xy,\\
\label{E6.3.1}\tag{E6.3.1}
\{z,x\}&=\Omega_y=xz,\\
\notag \{y,z\}&=\Omega_x=yz.
\end{align}
As before we assume that $a_1=0$. Note that  
$W:=\Bbbk xy+\Bbbk yz +\Bbbk xz$ which does not 
contain term $x^2$ and $y^2$. By \eqref{E6.0.1}, we have
$$(a_1x+a_2y+a_3z)y+(b_1x+b_2y+b_3z)x
=f(x)y+xf(y)=f(xy)\in W$$
which implies that 
$a_2=b_1=0$. Similarly, using $f(xz), f(yz)\in W$, we obtain 
that $a_3=c_1=b_3=c_2=0$. Thus $f(x)=0$, $f(y)=b_2 y$ and 
$f(z)=c_3 z$. Therefore $rgt(A)=-2$.

One can check that $\Kdim (A_{sing})=1$. 
\end{example}

\begin{example}
\label{xxex6.4}
Let $\Omega=xy(x+y)$. Then
\begin{align}
\notag \{x,y\}&=\Omega_z=0,\\
\label{E6.4.1}\tag{E6.4.1}
\{z,x\}&=\Omega_y=x^2+2xy,\\
\notag \{y,z\}&=\Omega_x=2xy+y^2.
\end{align}
Again we may assume that $a_1=0$. By \eqref{E6.0.1}, we have 
$$\begin{aligned}
f(x^2+2xy)&=2xf(x)+2xf(y)+2yf(x)\in 
\Bbbk (x^2+2xy)+\Bbbk (2xy+y^2)=:W,\\
f(2xy+y^2)&=2xf(y)+2yf(x)+2yf(y)\in 
\Bbbk (x^2+2xy)+\Bbbk (2xy+y^2).
\end{aligned}
$$
As a consequence, both $f(x)$ and $f(y)$ do not have the $z$ term, 
namely, $a_3=b_3=0$. Furthermore, by the above and a little bit of 
linear algebra, we have 
$$b_1=b_2=-a_2.$$
Now we can write $f(x)=a_2 y$ and $f(y)=-a_2 x-a_2y$. 
Applying $f$ to the second equation of \eqref{E6.4.1}, we obtain 
that $c_3=0$ and $a_2=0$ with $c_1$ and $c_2$ free. Therefore 
$rgt(A)=-2$.

One can check that $\Kdim (A_{sing})=1$. 
\end{example}

Next we consider the second case. Some linear algebra details will 
be omitted in the next two examples.

\begin{example}
\label{xxex6.5}
Suppose $\Omega=x\omega(x,y,z)$ where $\omega$ is an irreducible 
polynomial of degree $2$. Since $\Bbbk$ is algebraically closed, up 
to a linear change of variables, we can assume that $\omega$ is 
either $yz+ax^2+bxy+cxz$ (with further change of basis, we can 
assume that $c=0$) or $y^2+ ax^2+ bxy+ cxz$ (with further change 
of basis, we can assume that $b=0$). So we need to consider the 
following two cases:

Case 1: $\omega=yz+ ax^2+ b xy$. We can further assume $a=1$ since 
$\omega$ is irreducible. So $\omega=yz+x^2+b xy$. Replacing $z$ by 
$z-bx$, we have $\omega=yz+x^2$ and $\Omega=xyz+x^3$. In this case, 
the Poisson bracket of $A$ is determined by
\begin{align}
\notag \{x,y\}&=\Omega_z=xy,\\
\label{E6.5.1}\tag{E6.5.1}
\{z,x\}&=\Omega_y=xz, \\
\notag \{y,z\}&=\Omega_x=yz+3x^2.
\end{align}
One can check that $\Kdim (A_{sing})=1$. Recall that $W$ is 
$\{V,V\}=\Bbbk xy+ \Bbbk xz+\Bbbk (yz+3x^2)$ which does not 
involve either $y^2$ or $z^2$. By the second equation of 
\eqref{E6.5.1}, we have 
$$f(x)z +xf(z)\in W$$
which implies that $f(x)$ does not have the $z$ term, or $a_3=0$. 
Similarly, $b_3=0$ by the third equation of \eqref{E6.5.1}. 
By using the first equation of \eqref{E6.5.1}, we obtain that 
$a_2=0$. By \eqref{E6.0.3}, we can assume that $f(x)=0$. Now 
the first equation of \eqref{E6.5.1} implies that 
$\{x,f(y)\}=xf(y)$. So $f(y)\in \Bbbk y$ or $f(y)=b_2 y$. 

Using the second equation of \eqref{E6.5.1}, one can show that 
$c_1=c_2=0$. By using the third equation of \eqref{E6.5.1} and 
the fact that $W$ does not contain the term $y^2$, we obtain that
$c_2=0$. So $f(z)=c_3z$. From this we can derive that $b_2+c_3=0$.
Therefore $rgt(A)=-1$.

Case 2: $\omega=y^2+ ax^2+ c xz$. If $c=0$, then it is covered 
in case 1. If $c\neq 0$, we can assume $a=0$ and $c=1$. So 
$\omega=y^2+xz$ and $\Omega=xy^2+x^2z$. Then we have
\begin{align}
\notag \{x,y\}&=\Omega_z=x^2,\\
\label{E6.5.2}\tag{E6.5.2}
\{z,x\}&=\Omega_y=2xy,\\
\notag \{y,z\}&=\Omega_x=y^2+2xz.
\end{align}
One can check that $\Kdim (A_{sing})=1$. By definition 
$W=\Bbbk x^2+\Bbbk xy+\Bbbk (y^2+2xz)$, and it does not contain 
terms $z^2$ and $yz$. Using the third equation of \eqref{E6.5.2}, 
we have 
$$f(y^2+2xz)=2yf(y)+2xf(z)+2zf(x)\in W.$$
Therefore $f(x)z$ does not contain a $z^2$ term. So 
$f(x)=a_1 x+ a_2 y$ and with \eqref{E6.0.3} we can assume that 
$f(x)=a_2 y$. Now we apply $f$ to the first equation of 
\eqref{E6.5.2}, we obtain that $b_2=0$ and $b_3=-a_2$. (Some 
calculations are omitted.) Applying $f$ to the second equation 
of \eqref{E6.5.2}, we obtain that $a_2=0$, $c_3=0$ and 
$c_2=-2b_1$. Finally applying $f$ to the third equation of 
\eqref{E6.5.2}, we obtain that $c_1=0$ with $b_1$ free. Therefore 
$rgt(A)=-1$.
\end{example}

The final example deals with the irreducible cubic $\Omega$.

\begin{example}
\label{xxex6.6}
Suppose $\Omega$ is an irreducible cubic function in $x,y,z$. 
By classification (see, for example, \cite[Theorems 1 and 2]{KM} 
and \cite[Theorem 2.12]{BM}), there are following two singular 
ones and one non-singular.

Case 1: $\Omega=x^3+y^2z$. Then we have
\begin{align}
\notag \{x,y\}&=\Omega_z=y^2,\\
\label{E6.6.1}\tag{E6.6.1}
\{z,x\}&=\Omega_y=2yz,\\
\notag \{y,z\}&=\Omega_x=3x^2.
\end{align}
So $W:=\Bbbk y^2+\Bbbk yz+\Bbbk x^2$ does not have 
terms $z^2$, $xy$ and $xz$. Then 
$f(x^2)=2xf(x)\in W$ implies that $f(x)\in \Bbbk x$. 
By \eqref{E6.0.3}, we have $f(x)=0$. Applying 
$f$ to the first equation of \eqref{E6.6.1}, we obtain
that $f(y)=0$. Applying 
$f$ to the last two equations of \eqref{E6.6.1}, we obtain
that $f(z)=0$. So $rgt(A)=0$.

One can check that $\Kdim (A_{sing})=1$. By a Gr{\" o}bner Basis 
argument, one sees that the Hilbert series of $A_{sing}$ is 
$\frac{2}{(1-t)}+t^2+t-1$.

Case 2: $\Omega=x^3+x^2z+y^2z$. Then we have
\begin{align}
\notag \{x,y\}&=\Omega_z=x^2+y^2,\\
\label{E6.6.2}\tag{E6.6.2}
\{z,x\}&=\Omega_y=2yz,\\
\notag \{y,z\}&=\Omega_x=3x^2+2xz.
\end{align}
So $W$ does not have terms $z^2$ and $xy$.
By the second equation of \eqref{E6.6.2}, we have
$$f(yz)=yf(z)+zf(y)\in W$$
which implies that $f(y)$ has no $z$ term and $f(z)$
has no $x$ term. By the third relation of \eqref{E6.6.2}, 
we obtain that $f(x)$ has no $z$ term. By \eqref{E6.0.3}, 
one can assume that $f(x)=a_2 y$. Applying $f$ to the 
first equation, we obtain that $b_2=0$ and $b_1=-a_2$
(namely, $f(x)=a_2y$ and $f(y)=-a_2 x$). Applying $f$ to the 
second equation, we obtain that $c_2=a_2=0$ (so $f(x)=f(y)=0$). 
Then applying $f$ to the third equation of \eqref{E6.6.2} 
yields that $f(z)=0$. Therefore $rgt(A)=0$.

One can check that $\Kdim (A_{sing})=1$. By a Gr{\" o}bner Basis 
argument, one sees that the Hilbert series of $A_{sing}$ 
is $\frac{2}{(1-t)}+t^2+t-1$.

Case 3: $\Omega=\frac{1}{3}(x^3+y^3+z^3)+\lambda xyz$ where 
$\lambda^3\neq -1$ (which is the Hesse normal form given in 
\cite[Theorem, 2.12]{BM}). One can check that $A_{sing}$ is 
finite dimensional or $\Kdim (A_{sing})=0$. Consequently, 
$\Omega$ has an isolated singularity at zero. As mentioned 
at the end of Example \ref{xxex5.5}, we have $rgt(A)=0$ which 
follows from the Poisson cohomology computation given in
\cite[Proposition 4.5]{Pi1} (and \cite[Theorem 5.1]{VdB}). 
Here we will give a direct computation. By definition, 
\begin{align}
\notag \{x,y\}&=\Omega_z=z^2+\lambda xy,\\
\label{E6.6.3}\tag{E6.6.3}
\{z,x\}&=\Omega_y=y^2+\lambda xz,\\
\notag \{y,z\}&=\Omega_x=x^2+\lambda yz.
\end{align}
Note that $W=\Bbbk (z^2+\lambda xy)+\Bbbk (y^2+\lambda xz)
+\Bbbk (x^2+\lambda yz)$. This means that in $W$, $z^2$ 
(respectively, $y^2$ and $x^2$) appears together with 
$\lambda xy$ (respectively, $xz$ and $yz$). By the first 
equation of \eqref{E6.6.3}, we have $f(z^2+\lambda xy)\in W$. 
Using the notation in \eqref{E6.0.2}, we compute
$$\begin{aligned}
f(z^2+\lambda xy)
&= 2zf(z)+\lambda (xf(y)+yf(x))\\
&= 2z(c_1 x+c_2 y+c_3 z)+\lambda
[x(b_1 x+ b_2 y+b_3 z)+y(a_1 x+a_2 y+ a_3 z)]\\
&\equiv 2c_1 xz+ 2c_2 yz+ 2c_3 (-\lambda xy)
+\lambda [b_1(-\lambda yz)+b_2 xy+b_3 xz\\
&\qquad \qquad +a_1 xy
+a_2(-\lambda xz)+a_3 yz]
\mod W\\
&\equiv
(2c_1+\lambda b_3-\lambda^2 a_2) xz
+(2c_2-\lambda^2 b_1+\lambda a_3) yz
\\
&\qquad\qquad +(-2\lambda c_3+\lambda b_2 +\lambda a_1) xy 
\mod W.
\end{aligned}
$$
So we have
$$\begin{aligned}
2c_1+\lambda b_3-\lambda^2 a_2&=0,\\
2c_2+\lambda a_3-\lambda^2 b_1&=0,\\
-2\lambda c_3+\lambda b_2 +\lambda a_1&=0.
\end{aligned}
$$
From now on we assume that $\lambda\neq 0$ (if $\lambda=0$, 
the proof is slightly simpler and is omitted to save the space). 
With this assumption, we can remove $\lambda$ from the third 
equation of the above system. Similarly, by working with the 
last two equations in \eqref{E6.6.3}, we obtain the following
$$\begin{aligned}
2b_1+\lambda c_2-\lambda^2 a_3&=0,\\
2b_3+\lambda a_2-\lambda^2 c_1&=0,\\
-2b_2+c_3+a_1&=0,\\
2a_2+\lambda c_1-\lambda^2 b_3&=0,\\
2a_3+\lambda b_1-\lambda^2 c_2&=0,\\
-2a_1+b_2+c_3&=0.
\end{aligned}
$$
By \eqref{E6.0.3}, we may assume that $a_1=0$. Then by $3$ 
equations involving $a_1$, we obtain that $b_2=c_3=0$. Applying 
$f$ to three equations in \eqref{E6.6.3} with some linear 
algebra computations, we obtain that $f(x)=f(y)=f(z)=0$ 
(but this is true only if $\lambda^3\neq -1$). Therefore $rgt(A)=0$.

Since $\{\Omega_x,\Omega_y, \Omega_z\}$ is a regular sequence, 
the Hilbert series of $A_{sing}$ is $(1+t)^3$. 
\end{example}

Based on the above examples, we have the following
classification.

\begin{corollary}
\label{xxcor6.7} Let $A$ be a quadratic Poisson polynomial ring 
$\Bbbk[x,y,z]$. 
\begin{enumerate}
\item[(1)]
Suppose $A$ is unimodular. Then $(A,\Omega,rgt(A))$ is listed as follows 
up to isomorphisms

\[ \begin{array}{|c|l|l|l|l|l|l|l|} 
\hline \Omega & \; 0 \; &  x^3  &  x^2y &  xyz 
& xy(x+y)& xyz+x^3 & xy^2+x^2z\\ 
\hline rgt(A)    &-8 & -5  & -3  & -2 & -2     & -1   & -1\\ \hline 
    &  & {\text{Ex. 6.1}} &{\text{Ex. 6.2}}& {\text{Ex. 6.3}}& {\text{Ex. 6.4}} 
		&{\text{Ex. 6.5(1)}}& {\text{Ex. 6.5(2)}}\\ \hline 
\end{array} \]

\[ \begin{array}{|c|l|l|l|l|} 
\hline \Omega  & x^3+y^2 z &x^3+x^2z+y^2z
       & \frac{1}{3}(x^3+y^3+z^3) +\lambda xyz, \lambda^3\neq -1\\ 
\hline rgt(A) & 0         & 0           & 0\\ \hline 
& {\text{Ex. 6.6(1)}} & {\text{Ex. 6.6(2)}}&{\text{Ex. 6.6(3)}} \\ \hline
\end{array} \]
\item[(2)]
$Gpd(A)$ is one-dimensional if and only if $A$ is unimodular with $\Omega$ in the second table.
\end{enumerate}
\end{corollary}

\begin{definition}
\label{xxdef6.8}
A Poisson derivation $\phi$ of a Poisson algebra $A$ is called
{\it ozone} if $\phi(z)=0$ for all $z$ in the Poisson 
center of $A$.
\end{definition}

By Definition \ref{xxdef1.4}(1), the modular derivation ${\mathbf m}$ 
is always ozone.

\begin{lemma}
\label{xxlem6.9}
Let $A$ be the quadratic Poisson algebra in Example \ref{xxex6.6}(1)
with $\Omega=x^3+y^2z$. Then every ozone derivation of $A$ is 
Hamiltonian.
\end{lemma}

\begin{proof}
By definition,
\begin{align}
\{x,y\}&= y^2, \label{E6.9.1}\tag{E6.9.1}\\
\{z,x\}&= 2yz, \label{E6.9.2}\tag{E6.9.2}\\
\{y,z\}&= 3 x^2. \label{E6.9.3}\tag{E6.9.3}
\end{align}
We define a new grading on the polynomial ring $A=\Bbbk[x,y,z]$. 
Let $G$ be ${\mathbb Z}$ and define $\deg_G(x)=0$, $\deg_G(y)=1$ 
and $\deg_G(z)=-2$. For example, $\deg_G(\Omega)=0$. Every element 
$f\in A$ can be written as $\sum_{i\in {\mathbb Z}} f_{(i)}$ where 
$f_{(i)}$ is homogeneous of $G$-degree $i$. Then $f=f_{(i)}$ if 
and only if $f$ is homogeneous of $G$-degree $i$. By \eqref{E6.9.1} 
and \eqref{E6.9.2} the Hamiltonian derivation $H_x$ has $G$-degree 
1. 

\medskip

\noindent
{\bf Claim 1:} If $f$ is homogeneous of $G$-degree $i$, then 
$H_x( f)=i fy$.

\noindent
{\it Proof:} Let $f$ be a linear combination of monomials 
$x^a y^b z^c$. Since $\deg_G(f)=i$, we have $b-2c=i$. Then
$$\begin{aligned}
H_x(x^a y^b z^c)
&= b x^a y^{b-1} y^2 z^c + c x^a y^b z^{c-1} (-2yz)\\
&=(b-2c) (x^a y^b z^c)y=i (x^a y^b z^c)y.
\end{aligned}
$$
So the claim follows.

\medskip

Let $\phi$ denote an ozone Poisson derivation of $A$.

\medskip

\noindent
{\bf Claim 2:} Up to a Hamiltonian derivation, $\phi(x)=y w_{(0)}$
where $\deg_G(w_{(0)})=0$.

\noindent
{\it Proof:}
Since $\Omega$ is Poisson central, $\phi(\Omega)=0$. Then 
\begin{equation}
\label{E6.9.4}\tag{E6.9.4}
0=\phi(\Omega)=3x^2 \phi(x)+2yz \phi(y)+y^2 \phi(z).
\end{equation}
This implies that $y\mid \phi(x)$. Let $\phi(x)=y w$ where 
$w=\sum_{i\in {\mathbb Z}} w_{(i)}$ where $\deg_G(w_{(i)})=i$. 
By {\bf Claim 1}, $H_{\sum_{i\neq 0}\frac{-1}{i} w_{(i)}}(x)=
H_{x}(\sum_{i\neq 0}\frac{1}{i} w_{(i)})=\sum_{i\neq 0} 
w_{(i)} y$. After replacing $\phi$ by $\phi-H_{\sum_{i\neq 0}
\frac{-1}{i} w_{(i)}}$, we obtain that $\phi(x)=yw_{(0)}$ as 
required. 

\medskip

\noindent
{\bf Claim 3:} If $\deg_G(w)=0$, then $w$ is a polynomial
in $x$ and $\Omega$. 

\noindent
{\it Proof:} Since $w$ has $G$-degree 0, 
$w=\sum_{i,j\geq 0} \alpha_{i,j} x^{i} y^{2j} z^{j}$.
The assertion follows after replacing $y^2z$ by $\Omega-x^3$.

\medskip

From now on, we assume that $\phi(x)=y w_{(0)}$.

\medskip

\noindent
{\bf Claim 4:}
$\deg_G(\{\phi(x),y\})=3$ or $\{\phi(x),y\}=\{\phi(x),y\}_{(3)}$.

\noindent
{\it Proof:} Write $\phi(x)=y\sum_{i,k\geq 0} \alpha_{i,k} x^{i}\Omega^k$. 
We compute
$$\begin{aligned}
\{\phi(x),y\}&=\{y\sum_{k\geq 0} \alpha_k x^{n-3k}\Omega^k,y\}\\
&=y\{\sum_{i,k\geq 0} \alpha_{i,k} x^{i}\Omega^k,y\}\\
&=y\sum_{i,k\geq 0} \alpha_{i,k} i x^{i-1} y^2 \Omega^k
=y^3\sum_{i,k\geq 0} \alpha_{i,k} i x^{i-1} \Omega^k
\end{aligned}
$$
which has $G$-degree 3.

\medskip

\noindent
{\bf Claim 5:} $y^3\mid \phi(x)$.

\noindent
{\it Proof:} By {\bf Claim 2}, $\phi(x)=y\sum_{i,k\geq 0}
\alpha_{i,k} x^{i} y^{2k} z^k$. If $\alpha_{i,0}\neq 0$ for some 
$i$, we have a nonzero term $y x^{i+2}$ in $3x^2 \phi(x)$. But 
$y x^{i+2}$ cannot appear in $2yz \phi(y)+y^2 \phi(z)$ for any $i$, 
which contradicts \eqref{E6.9.4}. Therefore $\alpha_{i,0}=0$ for 
all $i$ and $y^3\mid \phi(x)$. 

\medskip

\noindent
{\bf Claim 6:} $y\mid \phi(y)$.

\noindent
{\it Proof:} This follows from \eqref{E6.9.4} and {\bf Claim 5}.

\medskip

It follows from \eqref{E6.9.4}, {\bf Claim 5}, and {\bf Claim 6} 
that
\begin{align}
\label{E6.9.5}\tag{E6.9.5}
\phi(x)&=yw_0=y^3z v_0,\\
\label{E6.9.6}\tag{E6.9.6}
\phi(y)&=y f,\\
\label{E6.9.7}\tag{E6.9.7}
\phi(z)&=-2zf-3x^2yz v_0.
\end{align}
where $v_0$ has $G$-degree 0 and $f\in A$. Next we will apply 
$\phi$ to the relations given in \eqref{E6.9.1} and \eqref{E6.9.2}. 
We compute
\begin{align}
\notag
0&=\phi(\{x,y\}-y^2)=\{\phi(x),y\}+\{x,\phi(y)\}-2y\phi(y)\\
\notag
&=\{y w_{(0)}, y\}+\{x, yf\}-2y\phi(y)\\
\notag
&=y\{w_{(0)},y\}_{(2)}+y^2 f+y\{x,f\}-2y^2 f\\
\notag
&=y\{w_{(0)},y\}_{(2)}-y^2 (\sum_{i\in {\mathbb Z}}f_{(i)})+
y\{x,\sum_{i\in {\mathbb Z}}f_{(i)}\}\\
\notag
&=y\{w_{(0)},y\}_{(2)}-y^2 (\sum_{i\in {\mathbb Z}}f_{(i)})+
y^2(\sum_{i\in {\mathbb Z}} i f_{(i)})\\
\notag
&=y\{w_{(0)},y\}_{(2)}+y^2(\sum_{i\in {\mathbb Z}} (i-1)f_{(i)}).
\end{align}
Therefore 
\begin{equation}
\label{E6.9.8}\tag{E6.9.8}
f=f_{(1)} \quad {\text{and}} \quad \{w_{(0)},y\}=0.
\end{equation} 
As a consequence, $f=y q_{(0)}$ where $\deg_G(q_{(0)})=0$.
So we have $\phi(y)=y^2 q_{(0)}$ and $\phi(z)=
-2yz q_{(0)}-3x^2 yz v_{(0)}$. 

Applying $\phi$ to \eqref{E6.9.2}, we have
\begin{align}
\notag
0&=\phi(\{z,x\}-2yz)=\{\phi(z),x\}+\{z,\phi(x)\}
-2\phi(y)z-2y\phi(z)\\
\notag
&=(-2q_{(0)}-3x^2 v_{(0)})\{yz,x\}+zv_{(0)}\{z,y^3\}
-2 y^2 q_{(0)} z-2y(-2yz q_{(0)}-3x^2 yz v_{(0)})\\
\notag
&=(-2q_{(0)}-3x^2 v_{(0)})y^2z+zv_{(0)}(-9x^2y^2)
-2 y^2 q_{(0)} z+4y^2z q_{(0)}+6x^2 y^2z v_{(0)}\\
\notag
&=-6x^2 y^2z v_{(0)}.
\end{align}
Therefore $v_{(0)}=0$ and 
\begin{align}
\label{E6.9.9}\tag{E6.9.9}
\phi(x)&=0,\\
\label{E6.9.10}\tag{E6.9.10}
\phi(y)&=y^2 q_{(0)},\\
\label{E6.9.11}\tag{E6.9.11}
\phi(z)&=-2yz q_{(0)}.
\end{align}
Write
$q_{(0)}=\sum_{i,k\geq 0} \alpha_{i,k} x^{i} \Omega^k$. Let 
$q'_{(0)}=\sum_{i,k\geq 0} \beta_{i,k} x^{i} \Omega^k$ where
$\beta_{i,k}:=\frac{\alpha_{i,k}}{i+1}$. It is easy to 
check that $H_{xq'_{(0)}}=\phi$. Therefore 
$\phi$ is Hamiltonian as desired.
\end{proof}

\begin{lemma}
\label{xxlem6.10}
Let $A$ be the quadratic Poisson algebra in Example 
\ref{xxex6.6}(2) with potential $\Omega=x^3+x^2z+y^2z$. 
Then every ozone derivation of $A$ is Hamiltonian.
\end{lemma}

\begin{proof} 
The Jacobian Poisson structure on $A=\kk[x,y,z]$ is explicitly 
given by
\begin{equation*}
\{x,y\}=x^2+y^2,\ \{y,z\}=3x^2+2xz,\ \{z,x\}=2yz.
\end{equation*}
We show that every ozone derivation of $A$ is Hamiltonian, 
which is based on a tedious computation. Since $A$ is graded, 
it suffices to check every graded Poisson derivation $\phi$ of 
degree $n$ vanishing on the Poisson center $Z$ is Hamiltonian. So 
we can write  
$$\phi(x)=\sum_{i=0}^{n+1} \phi_i x^i\in \kk[x,y,z] 
\text{ of degree $n+1$ with $\phi_i\in \kk[y,z]_{n+1-i}$. }$$

\medskip
\noindent
\textbf{Claim 1:} By subtracting a Hamiltonian derivation 
$H_g=\{g,-\}$ from $\phi$ for some suitable $g\in \kk[x,y,z]_n$, 
we can assume that $\phi_i\in \kk[z]\bigoplus y^3z(\kk [z,y^2z])$ 
for all $0\le i\le n-1$.

\noindent
{\it Proof:} For simplicity, we denote the $\kk$-linear map 
$T_m: \kk[y,z]_m\to \kk[y,z]_{m+1}$ by 
$$T_m(f):=2yz\frac{\partial{f}}{\partial z}-y^2
\frac{\partial{f}}{\partial y} \text{ for any $m\ge 0$.}$$
It is clear to check that 
\begin{itemize}
\item ${\rm ker}(T_m)=0$ if $3\nmid m$ and 
${\rm ker}(T_m)=\kk(y^2z)^{\frac{m}{3}}$ if $3\mid m$;
\item ${\rm img}(T_m)\bigoplus y({\rm ker}(T_m))=y(\kk[y,z]_m)$. 
In particular $y(y^2z)^{\frac{m}{3}}\not\in {\rm img}(T_m)$ 
if $3\mid m$.
\end{itemize}
For any homogeneous polynomial $g=\sum g_ix^i\in \kk[x,y,z]_n$ 
with $g_i\in \kk[y,z]_{n-i}$, we get
\begin{align*}
\{g,x\}&=\frac{\partial g}{\partial y}\{y,x\}
+\frac{\partial g}{\partial z}\{z,x\}
=-\frac{\partial g}{\partial y}(x^2+y^2)+\frac{\partial g}{\partial z}(2yz)\\
&=-\frac{\partial g_{n-1}}{\partial y}x^{n+1}
-\frac{\partial g_{n-2}}{\partial y}x^n+(T_1(g_{n-1})
-\frac{\partial g_{n-3}}{\partial y})x^{n-1}+\cdots\\
&\quad+(T_{n-2}(g_2)-\frac{\partial g_0}{\partial y})x^2
+T_{n-1}(g_1)x+T_n(g_0).
\end{align*}
Hence by choosing $g_0,g_1,\ldots,g_{n-1}$ for the coefficients 
of $x^0,x^1,\ldots, x^{n-1}$ inductively, we can achieve 
{\bf Claim 1}. 

\medskip

\noindent
\textbf{Claim 2:} Up to modulo a Hamiltonian derivation, we 
can set 
\begin{align*}
\phi(x)&=xzf(x,z)+y^3zg(x,y^2z)\\
\phi(y)&=(\frac{3}{2}xy+yz)f(x,z)
        +(\frac{3}{2}y^4-xy^2z)g(x,y^2z)+(x^2+y^2)p\\
\phi(z)&=-(3xz+2z^2)f(x,z)-3y^3zg(x,y^2z)-2yzp
\end{align*} 
where $f(x,z)\in \kk[x,z]_{n-1}$, $g(x,y^2z)\in \kk[x,y^2z]_{n-3}$ 
and $p\in \kk[x,y,z]_{n-1}$. Furthermore, we can assume $p$ does not 
contain $x^{n-1}$. 

\noindent
{\it Proof:} By {\bf Claim 1}, up to a Hamiltonian derivation, we 
can write 
\begin{align*}
\phi(x)=xzf(x,z)+y^3zg(x,y^2z)+az^{n+1}+bx^ny+cx^{n+1}
\end{align*}  
for some coefficients $a,b,c\in \kk$. Since $\kk[\Omega]\subseteq Z$, 
we have 
\begin{align*}
\phi(\Omega)=(3x^2+2xz)\phi(x)+2yz\phi(y)+(x^2+y^2)\phi(z)=0.
\end{align*}
We have three terms $2axz^{n+2}$, $3cx^{n+3}$ and $3bx^{n+2}y$ 
appearing in $(3x^2+2xz)\phi(x)$, which can only be canceled from the 
remaining terms in $2yz\phi(y)+(x^2+y^2)\phi(z)$. We get $a=0$ 
since $xz^{n+2}$ is certainly not contained there. For the term 
$3cx^{n+3}$, it has to be canceled by the terms in $(x^2+y^2)\phi(z)$, 
which implies that $\phi(z)=-3cx^{n+1}+\cdots$. But the another term 
$-3cx^{n+1}y^2$ appearing in $(x^2+y^2)\phi(z)$ has to be canceled 
inside the same product $(x^2+y^2)\phi(z)$. Repeating the argument, 
we get there is some $v\in \kk[x,y]$ such that $3cx^{n+3}+(x^2+y^2)v=0$. 
This is absurd unless $c=0$. The same argument shows that $b=0$ as well. 
So we can write $\phi(x)$ as in \textbf{Claim 2} and the expressions 
of $\phi(y)$ and $\phi(z)$ follow immediately. Finally, by further 
subtracting $\{a x^n,-\}$ from $\phi$, we can replace $p$ with 
$p-a n x^{n-1}$. So by choosing a suitable scalar $a$, we can 
assume $p$ does not contain $x^{n-1}$. 

\medskip

\noindent
\textbf{Claim 3}: we have
\begin{align*}
\{p,x\}&=-(2z+\frac{3}{2}x)f(x,z)-(3xz+2z^2)
\frac{\partial f(x,z)}{\partial z}+zx\frac{\partial f(x,z)}{\partial x}\\
&\quad +(3y^3-2xyz)g(x,y^2z)+y^3z\frac{\partial g(x,y^2z)}{\partial x}
-2xy^3z^2\frac{\partial g(x,y^2z)}{\partial (y^2z)}.
\end{align*}

\noindent
{\it Proof:} Since  $\phi$ is a Poisson derivation, we have
\begin{align*}
\phi(\{x,y\})=\{\phi(x),y\}+\{x,\phi(y)\}.
\end{align*}
A long and tedious calculation yields \textbf{Claim 3}. 

\medskip

Now we write 
$$f=\sum_{i=1}^n a_i x^{n-i}z^{i-1}\quad\text{and}\quad 
g=\sum_{i=1}^{\lfloor{\frac{n}{3}}\rfloor} b_ix^{n-3i}(y^2z)^{i-1}$$
for some $a_i,b_i\in \kk$ where we set $a_i=0$ if $i\not\in [1,n]$ 
and $b_i=0$ if $i\not\in [1,\lfloor\frac{n}{3}\rfloor]$. Define 
$$c_i:=(n-3i)a_i-(\frac{3}{2}+3i)a_{i+1}.$$
Then {\bf Claim 3} can be rewritten as 
\begin{align*}
\{p,x\}&=\sum_{i=0}^n c_iz^ix^{n-i}+\sum_{i=1}^{\lfloor 
\frac{n}{3}\rfloor}(3b_iy^{2i+1}z^{i-1})x^{n-3i}\\
&\quad -\sum_{i=1}^{\lfloor \frac{n}{3}\rfloor}
(2ib_iy^{2i-1}z^i)x^{n+1-3i}+\sum_{i=1}^{\lfloor 
\frac{n}{3}\rfloor}((n-3i)b_{i}y^{2i+1}z^i)x^{n-1-3i}.
\end{align*}
Write $p=\sum_{i=0}^{n-1}p_ix^i$ with $p_i\in \kk[y,z]_{n-1-i}$ 
and $p_{n-1}=0$. As in {\bf Claim 1}, we get 
\begin{align*}
\{p,x\}=\sum_{i=-1}^{n-1}\left(T_i(p_{n-1-i})-\frac{\partial 
p_{n-3-i}}{\partial y}\right)x^{n-1-i},
\end{align*}
where we set $T_{-1}=0$ and $p_{-2}=p_{-1}=p_n=0$. Hence we have

\begin{align}
T_{3i-1}(p_{n-3i})-\frac{\partial p_{n-3i-2}}{\partial y}
 &=c_{3i}z^{3i}+3b_iy^{2i+1}z^{i-1}\notag\\
\label{E6.10.1}\tag{E6.10.1}
T_{3i-2}(p_{n+1-3i})-\frac{\partial p_{n-1-3i}}{\partial y}
 &=c_{3i-1}z^{3i-1}-2ib_iy^{2i-1}z^i\\
T_{3i}(p_{n-1-3i})-\frac{\partial p_{n-3-3i}}{\partial y}
 &=c_{3i+1}z^{3i+1}+(n-3i)b_{i}y^{2i+1}z^{i}\notag
\end{align}
for all $1\le i\le \lfloor \frac{n}{3}\rfloor$ together with 
$-\frac{\partial p_{n-2}}{\partial y}=c_0$ and 
$-\frac{\partial p_{n-3}}{\partial y}=c_1z$. In particular, 
if $n\equiv 2 \pmod 3$, then $T_{n-1}(p_0)=c_nz^n$.

\medskip
\noindent
\textbf{Claim 4}: We have $f=0$.  

\noindent
{\it Proof:}
We show that $c_0=\cdots=c_n=0$, which implies that 
$a_1=\cdots=a_n=0$. By the above equations, it suffices to show 
that $T_{n-i-1}(p_i)-\frac{\partial p_{i-2}}{\partial y}\in y\kk[y,z]$ 
for $0\le i\le n$. We claim that $p_i\in \kk[y^2,z]$ for $0\le i\le n$. 
It is clear when $i=0$. Suppose it works for $p_i$ for all $i\le m$. 
Then inductively, the above equations imply that 
$T_{n-m-2}(p_{m+1})\in 
{\rm span}_\kk \{y^{2i+1}z^j\,|\, \text{for all possible}\ i,j\}$. 
Note that ${\rm ker}(T_{n-m-2})\in \kk[y^2,z]$. Our claim follows by the 
definition of $T_{n-m-2}$. Since ${\rm img}(T_{n-i-1})\in (y)$, we 
get $T_{n-i-1}(p_i)-\frac{\partial p_{i-2}}{\partial y}\in y\kk[y,z]$ 
for $0\le i\le n$.

\medskip

\noindent
\textbf{Claim 5}: We have $g=0$.  

\noindent
{\it Proof:} By {\bf Claim 4}, we can take $f=0$. We will only treat 
the case $n\equiv 0 \pmod 3$ here and other cases will follow in a
similar manner. We write $n=3s$ and group the equations 
\eqref{E6.10.1} into $s+1$ 
parts named by (Ei) with $0\le i\le s$. In details, (E0) is given by 
\begin{align*}
T_{n-1}(p_0)&=3b_sy^{2s+1}z^{s-1}\tag{E0.1}\\
T_{n-2}(p_1)&=-2sb_sy^{2s-1}z^{s}\tag{E0.2}\\
T_{n-3}(p_2)-\frac{\partial p_0}{\partial y}
&=3b_{s-1}y^{2s-1}z^{s-1}\tag{E0.3}
\end{align*}
For $1\le i\le s-2$, (Ei) is given by
\begin{align*}
T_{n-3i-1}(p_{3i})-\frac{\partial p_{3i-2}}{\partial y}
&=3b_{s-i}y^{2s-2i+1}z^{s-i-1}\tag{Ei.1}\\
T_{n-3i-2}(p_{3i+1})-\frac{\partial p_{3i-1}}{\partial y}
&=-2(s-i)b_{s-i}y^{2s-2i-1}z^{s-i}\tag{Ei.2}\\
T_{n-3i-3}(p_{3i+2})-\frac{\partial p_{3i}}{\partial y}
&=(3s-3(s-i-1))b_{s-i-1}y^{2s-2i-1}z^{s-i-1}\tag{Ei.3}
\end{align*}
Moreover, (E(s-1)) and (Es) are given by
\begin{align*}
T_{2}(p_{n-3})-\frac{\partial p_{n-5}}{\partial y}
&=3b_1y^3\tag{E(s-1).1}\\
T_{1}(p_{n-2})-\frac{\partial p_{n-4}}{\partial y}
&=-2b_1yz\tag{E(s-1).2}\\
-\frac{\partial p_{n-3}}{\partial y}
&=0\tag{E(s-1).3}\\
-\frac{\partial p_{n-2}}{\partial y}&=0\tag{Es}.
\end{align*}

As in {\bf Claim 4}, we know $p_i\in \kk[y^2,z]_{n-1-i}$. 
Assign the lexicographic order with $y>z$ on all monomials in $\kk[y,z]$. We 
prove the following statement inductively for all 
$0\le i\le s-2$ with $n=3s$:
\begin{align*}
p_{3i}&=-\frac{3}{2}b_{s-i}(y^2)^{s-i}z^{s-i-1}
       +\text{lower terms in $\kk[y^2,z]_{n-3i-1}$}\\
p_{3i+1}&=\alpha_i(y^2)^{s-i-1}z^{s-i}
        +\text{lower terms in $\kk[y^2,z]_{n-3i-2}$}\\
p_{3i+2}&=\beta_i(y^2)^{s-i-1}z^{s-i-1}
        +\text{lower terms in $\kk[y^2,z]_{n-3i-3}$}
\end{align*}
for some $\alpha_i,\beta_i\in \kk$ and $(i+1)b_{s-i-1}=(s-i)b_{s-i}$. 
When $i=0$, we use (E0.1) and (E0.2) to get 
$$p_0=-\frac{3}{2}b_sy^{2s}z^{s-1},\quad p_1=-sb_sy^{2s-2}z^s$$
since  $T_{n-1},T_{n-2}$ are injective.  So (E0.3) implies that 
$$
T_{n-3}(p_2)=3b_{s-1}y^{2s-1}z^{s-1}+\frac{\partial p_0}{\partial y}
=3(b_{s-1}-sb_s)y^{2s-1}z^{s-1}.
$$
Since $3\mid n-3$, we have ${\rm ker}(T_{n-3})=\kk y^{2s-2}z^{s-1}$ 
and $y^{2s-1}z^{s-1}\not\in {\rm img}(T_{n-3})$. We get $b_{s-1}=sb_s$ 
and $p_2=\beta_0y^{2s-2}z^{s-1}$ for some $\beta_0\in \kk$. Suppose 
the statement holds for $p_{3i},p_{3i+1},p_{3i+2}$. Then (E(i+1).1) 
implies that 
\begin{align*}
&T_{n-3i-4}(p_{3i+3})\\
=\,&3b_{s-i-1}y^{2s-2i-1}z^{s-i-2}+\frac{\partial }{\partial y}
 \left(\alpha_i(y^2)^{s-i-1}z^{s-i}
 +\text{lower terms in $\kk[y^2,z]_{n-3i-2}$}\right)\\
=\,&3b_{s-i-1}y^{2s-2i-1}z^{s-i-2}
 +\text{lower terms in $y(\kk[y^2,z]_{n-3i-4})$}.
\end{align*}
Since $T_{n-3i-4}$ is injective, we get 
\begin{align*}
p_{3i+3}=-\frac{3}{2}b_{s-i-1}(y^2)^{s-i-1}z^{s-i-2}
        +\text{lower terms in $\kk[y^2,z]_{n-3i-4}$}
\end{align*}
Similarly from (E(i+1).2) we get
\begin{align*}
p_{3i+4}=\alpha_{i+1}(y^2)^{s-i-2}z^{s-i-1}
        +\text{lower terms in $\kk[y^2,z]_{n-3i-5}$}
\end{align*}
for some $\alpha_{i+1}\in \kk$. Finally, (E(i+1).3) implies that 
\begin{align*}
T_{n-3i-6}(p_{3i+5})=&\,(3s-3(s-i-2))b_{s-i-2}y^{2s-2i-3}z^{s-i-2}\\
&\,+\frac{\partial }{\partial y}(-\frac{3}{2}b_{s-i-1}
(y^2)^{s-i-1}z^{s-i-2}+\text{lower terms in $\kk[y^2,z]_{n-3i-4}$})\\
=&\,3((i+2)b_{s-i-2}-(s-i-1)b_{s-i-1})y^{2s-2i-3}z^{s-i-2}\\
  &\,+\text{lower terms in $y(\kk[y^2,z]_{n-3i-6})$}
\end{align*}
Note that ${\rm ker}(T_{n-3i-6})=\kk y^{2s-2i-4}z^{s-i-2}$ and 
$y^{2s-2i-3}z^{s-i-2}\not\in {\rm img}(T_{n-3i-6})$ for 
$3\mid n-3i-6$. So we get $(i+2)b_{s-i-2}=(s-i-1)b_{s-i-1}$ and 
we can write  
$$
p_{3i+5}=\beta_{i+1}(y^2)^{s-i-2}z^{s-i-2}
        +\text{lower terms in $\kk[y^2,z]_{n-3i-6}$}
$$
for some $\beta_{i+1}\in \kk$. This completes our induction argument. 
From the above result, we have
$$
p_{n-5}=p_{3(s-2)+1}=\alpha_{s-2}y^2z^2+\text{lower terms in $\kk[y^2,z]_4$}
$$
From (E(s-1).1): $T_2(p_{n-3})-\frac{\partial p_{n-5}}{\partial y}=3b_1y^3$, 
we get $p_{n-3}=-\frac{3}{2}b_1y^2+\lambda z^2$ for some $\lambda \in\kk$. 
Moreover, (E(s-1).3): $-\frac{\partial p_{n-3}}{\partial y}=0$ implies 
that $b_1=0$. Again from the above statement, we have all $b_i=0$ and $g=0$. 

Finally, we can show that $\phi$ is Hamiltonian.  By all the above claims, 
up to a proper Hamiltonian derivation, we can take a Poisson derivation 
$\phi$ of degree $n$ as
$$
\phi(x)=0,\quad \phi(y)=(x^2+y^2)p,\quad  \phi(y)=-2yzp
$$
for some $p\in \kk[x,y,z]_{n-1}$. From 
$\{\phi(x),y\}+\{x,\phi(y)\}=\phi(x^2+y^2)$, we get $\{p,x\}=0$ 
or $(x^2+y^2)p_y=2yzp_z$. We show that $p=p(x,\Omega)$ by 
induction on the degree of $f$. It is clear that we can write 
$p_y=2yzq$ and $p_z=(x^2+y^2)q$ for some $q\in \kk[x,y,z]$ of 
degree ${\rm deg}(p)-3$. Then $p_{zy}=p_{yz}$ implies that 
$(x^2+y^2)q_y=2yzq_z$. So our induction hypothesis implies 
that $q=q(x,\Omega)$. Take any polynomial $h(x,\Omega)$ such 
that $\frac{\partial h(x,\Omega)}{\partial \Omega}=q$. An easy 
calculation shows that $h_y=p_y$ and $h_z=p_z$. So $p-h\in \kk[x]$. 
This proves our claim. Now take any $Q(x,\Omega)$ such that 
$\frac{\partial Q(x,\Omega)}{\partial x}=p(x,\Omega)$. 
Then one checks that $\phi=\{Q,-\}$ and $\phi$ is Hamiltonian.
\end{proof}

\begin{remark}
\label{xxrem6.11}
If $A$ is a non-unimodular quadratic Poisson polynomial algebra
$\Bbbk[x_1,\cdots,x_n]$, then by Corollary \ref{xxcor0.3},
$A^\delta$ is unimodular for some graded Poisson derivation 
$\delta$ of $A$ (in fact $\delta=\frac{1}{\sum_i \deg x_i} 
{\mathbf m}$). By \eqref{E4.3.1}, $rgt(A)=rgt(A^{\delta})$.
If one can calculate $rgt$ for all unimodular quadratic Poisson 
structures on $\Bbbk[x_1,\cdots,x_n]$, then the above formula 
provides a way of computing $rgt(A)$ when $A$ is not unimodular.

Note that all 13 classes of non-unimodular quadratic Poisson 
structures on $\Bbbk[x_1,x_2,x_3]$ were listed explicitly in 
\cite{DH} (also see \cite{DML, LX}). For each class, the 
modular derivation ${\mathbf m}$ is easy to compute. Therefore 
$rgt$ can be calculated by the method mentioned in the above
paragraph.
\end{remark}

\section{Rigidity, $H$-ozoneness, and $PH^1$-minimality}
\label{xxsec7}

In this section we will study some connections between rigidity 
of graded twisting, ozone derivations and the first Poisson 
cohomology.

Let $A$ be a general Poisson algebra with Poisson center $Z$.
Let $Pd(A)$ be the Lie algebra of all Poisson derivations 
of $A$ and let $Hd(A)$ be the Lie ideal of $Pd(A)$ of all 
Hamiltonian derivations. Recall from \eqref{E1.5.4} that the 
first Poisson cohomology of $A$ is defined to be
\begin{equation}
\label{E7.0.1}\tag{E7.0.1}
PH^1(A):=Pd(A)/Hd(A).
\end{equation}
If $A$ is ${\mathbb Z}$-graded, then so is $PH^1(A)$.
Part (1) of the following definition is Definition 
\ref{xxdef6.8}.

\begin{definition}
\label{xxdef7.1}
Let $A$ be a Poisson algebra.
\begin{enumerate}
\item[(1)]
A Poisson derivation $\phi$ of $A$ is called {\it ozone} 
if $\phi(z)=0$ for all $z\in Z$. 
\item[(2)]
Let $Od(A)$ denote the Lie algebra of all ozone Poisson 
derivations of $A$.
\item[(3)]
We say $A$ is {\it $H$-ozone} if $Od(A)=Hd(A)$, namely, 
if every ozone Poisson derivation is Hamiltonian.
\end{enumerate}
\end{definition}

It is clear that $Od(A)$ is a Lie ideal of $Pd(A)$ and 
$$Hd(A)\subseteq Od(A)\subseteq Pd(A).$$
In general, not every ozone Poisson derivation is 
Hamiltonian.

For the rest of this section, we only consider locally finite 
connected ${\mathbb N}$-graded Poisson algebras $A$ with $A_i\neq 0$
for some $i>0$. And later we will only consider $A=\Bbbk[x,y,z]$ 
where $\deg (x)=\deg (y)=\deg (z)=1$. In this case the Euler 
derivation, denoted by $E$, sending $a\to (\deg a)a(=:|a|a)$, is 
a nonzero Poisson derivation.

\begin{lemma}
\label{xxlem7.2}
Let $A$ be a connected graded Poisson algebra with center $Z$. 
Suppose $Z$ is a domain. Then $ZE\cap Od(A)=0$ if $Z\neq \Bbbk$. As a 
consequence, $ZE\cap Hd(A)=0$ and the canonical map 
$ZE\to PH^1(A)$ is a graded injective $Z$-module map.
\end{lemma}

\begin{proof} Let $f$ be any homogeneous element in $Z$.
It is easy to check that $fE$ is a Poisson derivation.
So $ZE$ is an abelian Lie subalgebra of $Pd(A)$. Moreover, 
one can check that $Pd(A)$ is a $Z$-module.

Next we assume that $Z\neq \Bbbk$. Let $\phi$ be in 
$ZE\cap Od(A)$ and we 
can write is as $\phi=fE$ for some $f\in Z$. Let $z\in Z$ be 
a nonzero element of positive degree. Then $\phi(z)=0$ as 
$\phi\in Od(A)$. Since $\phi=fE$, we obtain that 
$0=f(\deg z) z$. This implies that $f=0$ or $\phi=0$. Hence 
$ZE\cap Od(A)=0$.

Since $Hd(A)\subseteq Od(A)$, $ZE\cap Hd(A)=0$. So the map 
$$ZE\to Pd(A)/Hd(A)=:PH^1(A)$$ 
is injective.

Finally if $Z=\Bbbk$, then it is trivial since $E\not\in Hd(A)$. 
\end{proof}

By the above lemma, the minimal possibility of $PH^1(A)$ is $ZE$. 
This motivates the following definition.

\begin{definition}
\label{xxdef7.3}
Let $A$ be a nontrivial connected graded Poisson algebra with 
Poisson center $Z$. Suppose $Z$ is a domain. 
\begin{enumerate}
\item[(1)]
We say A is $PH^1$-minimal if $PH^1(A) \cong ZE$.
\item[(2)]
We say $A$ has an {\it Euler-ozone decomposition} if 
$$Pd(A)=ZE \rtimes Od(A).$$
\item[(3)]
We say $A$ has an {\it Euler-Hamiltonian decomposition} if 
$$Pd(A)=ZE \rtimes Hd(A).$$
\end{enumerate}
\end{definition}

By Remark \ref{xxrem5.3}, if $A$ is the Poisson polynomial 
algebra corresponding to the $4$-dimensional (resp. $3$-dimensional) 
Sklyanin algebra, then it is $PH^1$-minimal. 
Note that the dimension of $Gpd(A)$ is the constant term of 
$h_{Pd(A)}(t)$. So $A$ is rigid of graded twisting if 
and only if the constant term of $h_{Pd(A)}(t)$ is $1$
[Remark \ref{xxrem5.2}]. Therefore we have
\begin{equation}
\label{E7.3.1}\tag{E7.3.1}
{\text{$A$ is $PH^1$-minimal}} \quad
\Rightarrow \quad {\text{$rgt(A)=0$}}.
\end{equation}

\begin{proposition}
\label{xxpro7.4} 
Let $A$ be a connected graded Poisson algebra. 
Then the following are equivalent.
\begin{enumerate}
\item[(i)]
$A$ is $PH^1$-minimal.
\item[(ii)]
$h_{PH^1(A)}(t)=h_Z(t)$ provided $Z$ is a domain.
\item[(iii)]
$h_{Pd(A)}(t)=h_A(t)$.
\end{enumerate}
\end{proposition}

\begin{proof} 
(i) $\Leftrightarrow$ (ii) Follows from the definition.

(ii) $\Leftrightarrow$ (iii) It is clear that the map
$A\to Hd(A)$ sending $a\to H_a$ is surjective. The 
kernel is the center $Z$. So 
$h_{Hd(A)}(t)=h_{A}(t)-h_{Z}(t)$. By \eqref{E7.0.1},
$h_{PH^1(A)}(t)=h_{Pd(A)}(t)-h_{Hd(A)}(t)$.
Therefore
$$\begin{aligned}
h_{Pd(A)}(t)-h_A(t)
&=h_{PH^1(A)}(t)+h_{Hd(A)}(t)-h_A(t)\\
&=h_{PH^1(A)}(t)+h_{A}(t)-h_{Z}(t)-h_A(t)\\
&=h_{PH^1(A)}(t)-h_{Z}(t).
\end{aligned}
$$
The assertion follows.
\end{proof}

Let $A$ be the Poisson algebra in Example \ref{xxex6.6}(Case 3).
This Poisson algebra is corresponding to the $3$-dimensional
Sklyanin algebra. The potential $\Omega$ has isolated 
singularity. By \cite[Proposition 4.5]{Pi1} (and 
\cite[Theorem 5.1]{VdB}), $PH^1(A)=ZE$, consequently, $A$ 
is $PH^1$-minimal.

\begin{lemma}
\label{xxlem7.5}
Let $A$ be a connected graded Poisson algebra. Assume 
that $Z$ is a non-trivial domain. If $A$
is $PH^1$-minimal, then $A$ is $H$-ozone and has an 
Euler-Hamiltonian decomposition.
\end{lemma}

\begin{proof} Since $A$ is $PH^1$-minimal, 
$ZE\cong PH^1(A)=Pd(A)/Hd(A)$. This implies that 
$Pd(A)=ZE\rtimes Hd(A)$. So $A$ has an Euler-Hamiltonian 
decomposition. 

For a graded Poisson derivation $\phi$ of degree $d$, 
by the Euler-Hamiltonian decomposition, we have
$$\phi=fE+ H_a.$$
Let $z$ be a nonzero central element of positive degree. Then 
$\phi(z)=|z|fz+H_a(z)=|z|fz$. If $\phi$ is ozone, $0=\phi(z)=|z|fz$
which implies that $f=0$ and $\phi=H_a$ as required.
\end{proof}

\begin{lemma}
\label{xxlem7.6}
Let $A$ be a connected graded Poisson domain. Suppose $A$ is
$H$-ozone.
\begin{enumerate}
\item[(1)]
Every Poisson normal element in $A$ is Poisson 
central.
\item[(2)]
Suppose $A$ is a Poisson polynomial ring. Then 
$A$ is unimodular.
\end{enumerate}
\end{lemma}

\begin{proof}
(1) Let $x$ be a nonzero Poisson normal element. Then it is 
the sum of homogeneous Poisson normal elements. So we can 
assume that $x$ is homogeneous. Let $\phi$ be the log-Hamilton 
derivation $x^{-1}H_x$. Since $H_x(z)=0$ for all $z$ in the 
center $Z$, $\phi$ is ozone. By the hypothesis, $\phi$ is Hamiltonian, namely, $\phi=H_y$
for some element $y$. Since $\phi$ has degree 0, $\deg y=0$ (or 
$y\in\Bbbk$) and consequently, $\phi=0$. This implies that $x$ is
central.

(2) Let ${\mathbf m}$ be the modular derivation of $A$. It follows
from the definition that it is ozone. Since $\deg {\mathbf m}=0$, 
by the hypothesis, ${\mathbf m}=H_y$ for some element $y$ of degree 0. 
Hence $y\in\Bbbk$ and consequently, ${\mathbf m}=0$. The assertion 
follows.
\end{proof}

We have proved the following diagram

$$\begin{CD}
{\text{$A$ is $PH^1$-minimal}}
@>{\text{Lemma \ref{xxlem7.5}}}>>
{\text{$A$ is $H$-ozone}}\\
@V\eqref{E7.3.1} VV @VV {\text{Lemma \ref{xxlem7.6}(2)}}  V \\
{\text{$rgt(A)=0$}}
@>>{\text{Lemma \ref{xxlem4.4}(3)}}> 
{\text{$A$ is unimodular}}
\end{CD}
$$

Next we show that some of the conditions are equivalent 
under extra hypotheses.

\begin{lemma}
\label{xxlem7.7}
Suppose $A$ is a connected graded Poisson algebra. Let $Z$ be 
the center of $A$ and assume that $Z=\Bbbk[z]$ where $z$ is homogeneous with $\deg z>0$.
\begin{enumerate}
\item[(1)]
If $Pd(A)_{\leq -1}=0$, $A$ has an Euler-ozone decomposition.
\item[(2)]
Suppose $A$ is $H$-ozone. Then $rgt(A)=0$.
\end{enumerate}
\end{lemma}

\begin{proof} (1) Let $\phi$ be a Poisson derivation of $A$
of degree $i$. By the hypothesis, $i\geq 0$. 

Case 1: $\deg (z)\mid i$. Since $\phi(z)$ is central, $\phi(z)
=a z^n$ for some $a\in \Bbbk$ and $n\geq 0$. Then 
$\phi':=\phi-\frac{a}{|z|}z^{n-1} E$ satisfies $\phi'(z)=0$. So $\phi'$ 
is ozone. Therefore $\phi=\frac{a}{|z|}z^{n-1}E+\phi'$.

Case 2: $\deg (z)\nmid i$. Since $\phi(z)$ is central,
it must be $0$. Therefore $\phi$ is ozone.

Combining these two cases, every Poisson derivation is 
the sum of $fE$ for some $f\in Z$ and an ozone derivation.

(2) By Lemma \ref{xxlem4.4}(4), it suffices to show that
$\dim_{\Bbbk} Gpd(A)=1$, or equivalently, $Gpd(A)=\Bbbk E$.
Let $\phi\in Gpd(A)$ and $\phi(z)=az$ for some $a\in \Bbbk$.
Let $\delta$ be $\phi-\frac{a}{|z|}E$. Then $\delta\in Gpd(A)$ is 
ozone. By the hypothesis, $\delta$ is Hamiltonian, say
$\delta=H_f$ for some homogeneous element $f\in A$. Since 
$\deg (\delta)=0$, $\deg(f)=0$. Since $A$ is connected graded,
$H_f=0$ and consequently, $\delta=0$. Thus $\phi=\frac{a}{|z|}E$ and 
$Gpd(A)=\Bbbk E$.
\end{proof}

Now we are ready to prove Theorem \ref{xxthm0.6}.

\begin{proof}[Proof of Theorem \ref{xxthm0.6}]
(1) $\Rightarrow$ (2): Since $rgt(A)=0$, every 
Poisson derivation $\delta$ is of the form $cE$. 
Then $E\wedge \delta=0$. By \eqref{E2.3.1}, 
$\langle a,b \rangle =\{a,b\}$. So $A=A^{\delta}$.
The assertion follows. 

(2) $\Rightarrow$ (1): By Corollary \ref{xxcor0.3}, there 
is a Poisson derivation $\delta$ such that $A^{\delta}$ is 
unimodular. Since $A^{\delta}\cong A$ for all $\delta$, $A$ 
is unimodular. Suppose to the contrary that $A$ is not rigid. 
Then there is a Poisson derivation $\delta$ not in $ZE$.
Thus, by Theorem \ref{xxthm0.2}, the modular derivation of 
$A^{\delta}$ is
$${\mathbf n}=0+(\sum_{i=1}^3 \deg(x_i))\delta-\divv(\delta) E$$
which cannot be zero as $\divv(\delta)\in \Bbbk$ [Lemma
\ref{xxlem1.2}(3)]. Therefore $A^{\delta}$ is not 
isomorphic to $A$, yielding a contradiction.

(5) $\Leftrightarrow$ (6): Under the hypothesis of (5), $A$ is 
$PH^1$-minimal. One implication follows by Lemma \ref{xxlem7.5}
and the other is clear.

(6) $\Rightarrow$ (7): See the proof of Lemma \ref{xxlem7.5}.

(7) $\Rightarrow$ (1): This is Lemma \ref{xxlem7.7}(2).

(3) $\Leftrightarrow$ (5): This is Proposition \ref{xxpro7.4}.

(1) $\Leftrightarrow$ (8): This is Corollary \ref{xxcor6.7}.

(8) $\Rightarrow$ (6,7): If $\Omega=\frac{1}{3}(x^3+y^3+z^3)
+\lambda xyz$ with $\lambda^3\neq -1$, then by the comments before
Lemma \ref{xxlem7.5}, $A$ is $PH^1$-minimal.
Hence $A$ is $H$-ozone since (5) $\Leftrightarrow$ (6). If
$\Omega$ is $x^3+y^2z$ or $x^3+x^2z+y^2z$, it follows
by Lemmas \ref{xxlem6.9} and \ref{xxlem6.10} that $A$ is $H$-ozone.
In all three cases in Example \ref{xxex6.6}, one can check
easily that $Pd(A)_{\leq -1}=0$. By Lemma \ref{xxlem7.7}(1),
$A$ has an Euler-ozone decomposition. Since $A$ is $H$-ozone,
$A$ has an Euler-Hamiltonian decomposition. 

(5) $\Rightarrow$ (1): The assertion follows from Remark 
\ref{xxrem5.2}.

(4) $\Rightarrow$ (1): The assertion follows from Remark 
\ref{xxrem5.2}

(8) $\Rightarrow$ (4): In all three cases, $Z$ is 
$\Bbbk[\Omega]$ (this is a well-known fact and a special case of it 
is \cite[Lemma 1]{MTU}, also see Lemmas \ref{xxlem7.8} and 
\ref{xxlem7.9} later), which has Hilbert series $\frac{1}{1-t^3}$.
The assertion follows from (5) since (5) is equivalent to
(8).

(4) $\Leftrightarrow$ (9): It follows from \eqref{E1.5.5}
that $h_{PH^3(A)}(t)-h_{PH^2(A)}(t)=h_{PH^0(A)}(t)-h_{PH^1(A)}(t)+t^{-3}$.
We know that $PH^0(A)=Z$. So the assertion follows 
from the fact that $h_{PH^1(A)}(t)=h_{PH^0(A)}(t)=\frac{1}{1-t^3}$
if and only if $h_{PH^3(A)}(t)-h_{PH^2(A)}(t)=t^{-3}$.
\end{proof}

Before we prove Corollary \ref{xxcor0.7}, we will need to 
calculate the Hilbert series of $PH_0(A)$ for $A$ in the 
first two cases of Example \ref{xxex6.6}.  By definition 
and \cite[p.2357]{Pe2}, the 0th Poisson homology 
of the Poisson polynomial algebra $\Bbbk[x,y,z]$ is 
\begin{equation}
\notag
PH_0(A)\cong \frac{A}{\{A,A\}}=\frac{A}{(H_x(A)+H_y(A)+H_z(A))}.
\end{equation}

Case 1: $\Omega=x^3+y^2z$. We use the $G$-grading introduced
in the proof of Lemma \ref{xxlem6.9}, namely, $\deg_G(x)=0,
\deg_G(y)=1$ and $\deg_G(z)=-2$. Let 
$$\begin{aligned}
a_{i,j,0,l}&=x^i y^j z^0 \Omega^l, \quad i,j,l\geq 0,\\
b_{i,0,k,l}&=x^i y^0 z^k \Omega^l, \quad i,l\geq 0, k\geq 1,\\
c_{i,1,k,l}&=x^i y   z^k \Omega^l, \quad i,l\geq 0, k\geq 1,\\
{\mathbb A}&:=\{a_{i,j,0,l} \mid i,j,l\geq 0\},\\
{\mathbb B}&:=\{b_{i,0,k,l} \mid i,l\geq 0, k\geq 1\},\\
{\mathbb C}&:=\{c_{i,1,k,l} \mid i,l\geq 0, k\geq 1\}.
\end{aligned}
$$

If $X$ is a subset of elements in $A$, we use $\Bbbk X$ to 
denote the $\Bbbk$-linear span of $X$. 

\begin{lemma}
\label{xxlem7.8} Retain the above notations.
\begin{enumerate}
\item[(1)]
${\mathbb A}\cup {\mathbb B}\cup {\mathbb C}$ is a $\Bbbk$-linear
basis of $A$.
\item[(2)]
$$\begin{aligned}
H_x(a_{i,j,0,l})&=ja_{i,j+1,0,l},\\
H_x(b_{i,0,k,l})&=(-2k)c_{i,1,k,l},\\
H_x(c_{i,1,k,l})&=
\begin{cases} (1-2k)(b_{i,0,k-1,l+1}-b_{i+3,0,k-1,l}) & k>1,\\
(1-2k) (a_{i,0,0,l+1}-a_{i+3,0,0,l}) & k=1.\end{cases}
\end{aligned}
$$
\item[(3)]
$$\begin{aligned}
H_y(a_{i,j,0,l})&=(-i) a_{i-1,j+2,0,l},\\
H_y(b_{i,0,k,l})&=
\begin{cases} (-i) b_{i-1,0,k-1,l+1}+(i+3k)b_{i+2,0,k-1,l} &k>1,\\
(-i) a_{i-1,0,0,l+1}+(i+3)a_{i+2,0,0,l} &k=1,\end{cases}\\
H_y(c_{i,1,k,l})&=
\begin{cases} (-i) c_{i-1,1,k-1,l+1}+(i+3k)c_{i+2,1,k-1,l} &k>1,\\
(-i) a_{i-1,1,0,l+1}+(i+3)a_{i+2,1,0,l} &k=1.\end{cases}
\end{aligned}
$$
\item[(4)]
$$\begin{aligned}
H_z(a_{i,j,0,l})&=
\begin{cases} 2ia_{i-1,j-1,0,l+1}+(-2i-3j)a_{i+2,j-1,0,l} & j>0,\\
2ic_{i-1,1,1,l} & j=0,\end{cases}\\
H_z(b_{i,0,k,l})&=2i c_{i-1,1,k+1,l},\\
H_z(c_{i,1,k,l})&=2ib_{i-1,0,k,l+1}+
(-3-2i)b_{i+2,0,k,l}.
\end{aligned}
$$
\item[(5)]
$A/(H_x(A)+H_y(A)+H_z(A))$ has a $\Bbbk$-linear basis
$$\{a_{0,0,0,0},a_{1,0,0,0}\}\cup
\{a_{0,1,0,l}\}_{l\ge 0}\cup \{a_{1,1,0,l}\}_{l\ge 0} \cup \{b_{0,0,k,0}\}_{k\geq 1}
\cup\{b_{1,0,k,0}\}_{k\geq 1}.$$
\item[(6)]
The Hilbert series of $PH_0(A)$ is 
$$\frac{(1+t)^3}{1-t^3}.$$
\item[(7)]
The Hilbert series of $PH^0(A)=h_Z(t)$ is 
$$\frac{1}{1-t^3}.$$
\end{enumerate}
\end{lemma}

\begin{proof} This follows from a tedious and direct computation.
\end{proof}

The proof of the above lemma is routine and long, but very 
elementary, only using easy linear algebra 
arguments. To save space the details are omitted here. Note that 
Lemma \ref{xxlem7.8}(7) is a well-known fact.

\medskip

Case 2: $\Omega=x^3+x^2z+y^2z$. We need to prove a lemma 
similar to Lemma \ref{xxlem7.8}. We use the same notations
as in Case 1 except that $\Omega$ is $x^3+x^2z+y^2z$ instead of
$x^3+y^2z$.

\begin{lemma}
\label{xxlem7.9} Let $A$ be as in Example \ref{xxex6.6}(Case 2)
with potential $\Omega=x^3+x^2z+y^2z$. 
\begin{enumerate}
\item[(1)]
${\mathbb A}\cup {\mathbb B}\cup {\mathbb C}$ is a $\Bbbk$-linear
basis of $A$.
\item[(2)]
$$\begin{aligned}
H_x(a_{i,j,0,l})&=ja_{i,j+1,0,l}+ja_{i+2,j-1,0,l},\\
H_x(b_{i,0,k,l})&=(-2k)c_{i,1,k,l},\\
H_x(c_{i,1,k,l})&=
\begin{cases} (1-2k)(b_{i,0,k-1,l+1}-b_{i+3,0,k-1,l}) +2kb_{i+2,0,k,l}& k>1,\\
-(a_{i,0,0,l+1}-a_{i+3,0,0,l})+2b_{i+2,0,1,l} & k=1.\end{cases}
\end{aligned}
$$
\item[(3)]
$$\begin{aligned}
H_y(a_{i,j,0,l})&=(-i) a_{i-1,j+2,0,l}+(-i)a_{i+1,j,0,l},\\
H_y(b_{i,0,k,l})&=
\begin{cases} (-i) b_{i-1,0,k-1,l+1}+(i+3k)b_{i+2,0,k-1,l}+2kb_{i+1,0,k,l} &k>1,\\
(-i) a_{i-1,0,0,l+1}+(i+3)a_{i+2,0,0,l}+2b_{i+1,0,1,l} &k=1,\end{cases}\\
H_y(c_{i,1,k,l})&=
\begin{cases} (-i) c_{i-1,1,k-1,l+1}+(i+3k)c_{i+2,1,k-1,l}+2kc_{i+1,1,k,l}&k>1,\\
(-i) a_{i-1,1,0,l+1}+(i+3)a_{i+2,1,0,l}+2c_{i+1,1,1,l} &k=1.\end{cases}
\end{aligned}
$$
\item[(4)]
$$\begin{aligned}
H_z(a_{i,j,0,l})&=
\begin{cases} 2ic_{i-1,1,1,l} & j=0,\\
2ia_{i-1,0,0,l+1}+(-2i-3)a_{i+2,0,0,l}+(-2i-2)b_{i+1,0,1,l}& j=1,\\
2ia_{i-1,1,0,l+1}+(-2i-6)a_{i+2,1,0,l}+(-2i-4)c_{i+1,1,1,l}& j=2,\\
2ia_{i-1,j-1,0,l+1}+(-2i-3j)a_{i+2,j-1,0,l} & \\
\qquad\qquad +(-2i-2j)x^{i+1}y^{j-1}z\Omega^l& j\geq 3,\\
\end{cases}\\
H_z(b_{i,0,k,l})&=2i c_{i-1,1,k+1,l},\\
H_z(c_{i,1,k,l})&=2ib_{i-1,0,k,l+1}+
(-3-2i)b_{i+2,0,k,l}+(-2-2i)b_{i+1,0,k+1,l}.
\end{aligned}
$$
\item[(5)]
$A/(H_x(A)+H_y(A)+H_z(A))$ has a $\Bbbk$-linear basis
$$\{a_{0,0,0,0},a_{1,0,0,0}\}\cup \{a_{3i,1,0,0}\}_{i\ge 0} 
\cup \{a_{1+3i,1,0,0}\}_{i\ge 0}\cup \{b_{0,0,k,0}\}_{k\geq 1}
\cup\{b_{1,0,k,0}\}_{k\geq 1}.$$
\item[(6)]
The Hilbert series of $PH_0(A)$ is 
$$\frac{(1+t)^3}{1-t^3}.$$
\item[(7)]
The Hilbert series of $PH^0(A)=h_Z(t)$ is 
$$\frac{1}{1-t^3}.$$
\end{enumerate}
\end{lemma}

\begin{proof} This follows from a tedious and direct computation.
\end{proof}

Similar to Lemma \ref{xxlem7.8}, the proof of Lemma \ref{xxlem7.9} 
is routine and long (even longer than the proof of Lemma \ref{xxlem7.8}), 
but still very elementary. To save space the details are omitted here.
Note that Lemma \ref{xxlem7.9}(7) is a well-known fact.

Finally we prove Corollary \ref{xxcor0.7}. 

\begin{proof}[Proof of Corollary \ref{xxcor0.7}]
(1) This is clear since $Z=\Bbbk[\Omega]$ by 
Lemmas \ref{xxlem7.8} and \ref{xxlem7.9}.

(2) Follows from Theorem \ref{xxthm0.6}((8)$\Rightarrow$(4)).

(4) By Poincar{\' e} duality \cite[Theorem 3.5]{LuWW1}, $h_{PH^3(A)}(t)=
t^{-3}h_{PH_0(A)}$. Then the assertion follows from 
Lemmas \ref{xxlem7.8} and \ref{xxlem7.9}.

(3) Follows from Parts (1,2,4) and \eqref{E1.5.5}.
\end{proof}

\subsection*{Acknowledgments}
Wang was partially supported by Simons collaboration grant 
\#688403 and Air Force Office of Scientific Research grant 
FA9550-22-1-0272. Zhang was partially supported by 
the US National Science Foundation (No. DMS-1700825 and 
DMS-2001015). Part of this research work was done during 
the first and second authors' visit to the Department of 
Mathematics at University of Washington in January 2022. 
They are grateful for the third author’s invitation and 
wish to thank University of Washington for its 
hospitality.

\providecommand{\bysame}{\leavevmode\hbox to3em{\hrulefill}\thinspace}
\providecommand{\MR}{\relax\ifhmode\unskip\space\fi MR }
\providecommand{\MRhref}[2]{%

\href{http://www.ams.org/mathscinet-getitem?mr=#1}{#2} }
\providecommand{\href}[2]{#2}

\end{document}